\documentclass{amsart}
\usepackage{amsmath,amscd,xypic,amssymb,combelow,color,enumitem,graphicx,float,comment,mathtools}
\usepackage{tikz-cd}
\usepackage[hidelinks]{hyperref}
\definecolor{citation}{rgb}{0,.40,.80}
\hypersetup{
   colorlinks=true,
   citecolor=black,
    linkcolor=citation,
    }
\usepackage[style=alphabetic, backend=biber, eprint=true, maxbibnames=99, doi=false, url = false, isbn=false]{biblatex}

\newtheorem{theorem}{Theorem}[section]

\newtheorem{proposition}[theorem]{Proposition}

\newtheorem{assumption}[theorem]{Assumption}
\newtheorem{corollary}[theorem]{Corollary}
\newtheorem{conjecture}[theorem]{Conjecture}
\newtheorem{definition}[theorem]{Definition}

\newtheorem{example}[theorem]{Example}
\newtheorem{remark}[theorem]{Remark}

\def\fg{{\mathfrak{g}}}

\def\GL{\mathrm{GL}}

\def\BC{{\mathbb{C}}}
\def\BK{{\mathbb{K}}}

\def\BN{{\mathbb{N}}}

\def\BQ{{\mathbb{Q}}}
\def\BZ{{\mathbb{Z}}}

\def\Gr{\mathrm{Gr}}

\DeclareMathOperator{\Tr}{\mathrm{Tr}}
\def\vir{\mathrm{vir}}

\def\id{\textrm{id}}
\DeclareMathOperator{\ptau}{{}^{\mathfrak{p}}\tau}
\DeclareMathOperator{\JH}{\mathrm{JH}}
\DeclareMathOperator{\Per}{\mathrm{Perv}}

\DeclareMathOperator{\BCu}{B\mathbb{C}^{*}}
\DeclareMathOperator{\act}{\mathrm{act}}
\DeclareMathOperator{\BT}{\mathrm{BT}}
\DeclareMathOperator{\HHT}{\mathrm{H}^{T}}
\DeclareMathOperator{\HH}{\mathrm{H}}

\def\CA{{\mathcal{A}}}
\def\CB{{\mathcal{B}}}

\def\CS{{\mathcal{S}}}

\def\CV{{\mathcal{V}}}

\def\sym{\textrm{sym}}
\def\Sym{\textrm{Sym}}

\def\bs{\boldsymbol{\varsigma}}

\def\nn{{\mathbb{N}^I}}
\def\zz{{\mathbb{Z}^I}}

\def\bm{{\boldsymbol{m}}}
\def\bn{{\boldsymbol{n}}}

\def\bt{{\boldsymbol{t}}}
\def\bu{{\boldsymbol{u}}}

\def\b0{{\boldsymbol{0}}}

\def\loccit{\emph{loc.~cit.~}}

\def\Spec{\text{Spec}}

\def\oij{i \rightarrow j}
\def\oji{j \rightarrow i}

\def\oCS{\mathring{\CS}}

\def\Sym{\text{Sym}}

\def\op{\text{op}}

\def\oCS{\mathring{\CS}}

\def\edge{\Omega}

\def\dQ{\bar{Q}}
\def\tQ{\tilde{Q}}

\def\tW{\tilde{W}}

\newcommand{\pD}{{^{\mathfrak{p}}\!\mathcal{D}}}
\newcommand{\pH}{{^{\mathfrak{p}}\!\mathcal{H}}}

\begin{document}

\title[BPS Lie algebras, perverse filtrations and shuffle algebras]{\Large{\textbf{BPS Lie algebras, perverse filtrations and shuffle algebras}}} 

\author[Shivang Jindal and Andrei Negu\cb t]{Shivang Jindal and Andrei Negu\cb t}

\address{École Polytechnique Fédérale de Lausanne (EPFL), Lausanne, Switzerland}
\email{shivang.jindal@epfl.ch}

\address{École Polytechnique Fédérale de Lausanne (EPFL), Lausanne, Switzerland \newline \text{ } \ \ Simion Stoilow Institute of Mathematics (IMAR), Bucharest, Romania} 
\email{andrei.negut@gmail.com}

\maketitle
	
\begin{abstract} We give an explicit description of the BPS Lie algebra of any quiver with zero potential, by relating the perverse filtration on the cohomological Hall algebra with certain limit conditions on polynomials. Our results also give a partial description of the perverse filtration for arbitrary potential, which we conjecture is complete in the case of tripled quivers with canonical cubic potential.

\end{abstract}

\bigskip
\section{Introduction}
\label{sec:intro}

\subsection{Cohomological Hall algebras} 

Ever since their introduction in \cite{KS}, cohomological Hall algebras (CoHAs for short) have been the natural setting for the study of Donaldson-Thomas theory for quivers with potential. Specifically, to any quiver (i.e. oriented graph) $Q$ and potential $W$ (i.e. linear combination of oriented cycles of $Q$) one may associate an algebra
$$
\CA^T_{Q,W}
$$
using the moduli stacks of representations of $Q$, sheaves of vanishing cycles associated to the potential $W$, and equivariant with respect to a torus $T$ as in Subsection \ref{sub:quiver}. A particularly important case of the construction above is when the potential is 0, in which case \cite{KS} showed that
\begin{equation}
\label{eqn:iso intro}
\CA^T_Q := \CA^T_{Q,0} \cong \bigoplus_{\bn = (n_i \geq 0)_{i \in I}} \BC[\text{Lie }T][z_{i1},\dots,z_{in_i}]^{\text{sym}}_{i \in I}
\end{equation}
made into an algebra with respect to the shuffle product \eqref{eqn:mult}; above, $I$ denotes the vertex set of $Q$. For any potential $W$, then under the mild assumption of Subsection \ref{morphismshuffle}, there is an algebra homomorphism 
\begin{equation}
\label{eqn:map to shuffle intro}
\CA^T_{Q,W} \xrightarrow{\Psi} \CA^T_{Q}
\end{equation}
which is known to be injective in many cases of interest (see Section \ref{sec:triple}). Therefore, the shuffle product is a useful explicit model for the study of CoHAs. 

\medskip

\subsection{Perverse filtration} 

An important feature of cohomological Hall algebras is their perverse filtration 
$$
P^1(\CA^T_{Q,W}) \subset P^2 (\CA^T_{Q,W}) \subset \dots \subset \CA^T_{Q,W}
$$
which we will recall in Subsection \ref{sub:perverse}. This filtration is defined in \cite{davison2020cohomological} by considering the pushforward of the semisimplification morphism from the moduli stack of quiver representations to its coarse moduli space, and captures important topological invariants. In recent years, perverse filtrations defined in this way have played a prominent role in geometric representation theory \cite{botta2023okounkovs, achar} and enumerative geometry \cite{todamaulik}. 

\medskip

Thus, it becomes quite relevant to obtain an explicit description of the perverse filtration on $\CA^T_Q$, in terms of the polynomial rings which appear in \eqref{eqn:iso intro}. In the present paper, we precisely give such a description, loosely inspired by the case of the Jordan quiver treated in \cite{NExts}. Explicitly, for any $d \geq 1$ define
$$
F^d ( \CA^T_{Q} )
$$
to be the set of polynomials $E(z_{i1},\dots,z_{in_i})_{i \in I} \in \CA^T_{Q}$ such that for any partition
\begin{equation}
\label{eqn:intro partition}
\bn = \bn^1 + \dots + \bn^k
\end{equation}
(for arbitrary $k \geq 1$ and $\bn^1,\dots,\bn^k \in \nn \backslash \b0$, where $\bn = (n_i)_{i \in I}$), the polynomial
\begin{equation}
\label{eqn:intro add y}
E(y_1+z_{i,1},\dots,y_1+z_{i,n_i^1}, \dots,y_k+z_{i,n_i-n_i^k+1}, \dots, y_k+z_{i,n_i})_{i \in I}
\end{equation}
satisfies the degree bound (in $y_1,\dots,y_k$)
\begin{equation}
\label{eqn:intro degree bound}
\text{deg}_{y_1,\dots,y_k}(E) \leq \frac {d-k}2 - \sum_{1\leq a < b \leq k} (\bn^a,\bn^b)
\end{equation}
(the symmetric bilinear form will be defined in \eqref{eqn:euler form}). We note that the condition above is somewhat different from the one that naturally appears in $K$-theoretic Hall algebras (cf. \cite{NShuffle}), wherein one constrains the monomials $\prod_{i,a} z_{i,a}^{m_{i,a}}$ appearing in $E$ by a collection of linear inequalities satisfied by the exponents $m_{i,a}$.

\medskip 

\begin{theorem}
\label{thm:main}

The filtration 
$$
F^1(\CA^T_{Q}) \subset F^2(\CA^T_{Q}) \subset \dots \subset \CA^T_{Q}
$$
defined above matches with the perverse filtration, i.e.
\begin{equation}
\label{eqn:intro main}
P^d (\CA^T_Q) = F^d (\CA^T_{Q})
\end{equation}
for all $d \geq 1$.

\end{theorem}

\medskip 

In particular, since the BPS Lie algebra $\mathfrak{g}^T_{Q,W}$ is defined to be $P^1(\mathcal{A}^T_{Q,W})$, the $d=1$ case of \eqref{eqn:intro main} shows that the BPS Lie algebra of a quiver with 0 potential is completely determined by the $d=1$ particular case of the degree bounds \eqref{eqn:intro degree bound} (see Corollary \ref{BPS Lie algebra} for the precise statement incorporating sign twists). For an arbitrary torus $T$, the Lie algebra structure on $\fg^T_{Q}$ is far from being known. Even in the case of trivial torus, the exact characterization of the resulting BPS Lie algebra $\fg_Q$ as a subspace of polynomial rings is not known. \medskip

Our next result shows that the knowledge of perverse filtration on the cohomological Hall algebra $\mathcal{A}^T_Q$ with 0 potential can be used to deduce properties of the perverse filtration on the cohomological Hall algebra $\mathcal{A}^T_{Q,W}$ with arbitrary potential. It arose from conversations between the first-named author and Lucien Hennecart, to whom we are grateful.

\medskip

\begin{theorem}[Proposition \ref{prop:perverse}]\label{hommorphismcompatibility}
Let $Q,W,T$ satisfy Assumption \ref{assumption1}, i.e  $\mathcal{A}^T_{Q,W}$ is a free module over the cohomology ring $\mathbb{C}[\mathrm{Lie \ }T] $. Then the morphism $\Psi: \mathcal{A}^T_{Q,W} \rightarrow \mathcal{A}^T_Q$ of \eqref{eqn:map to shuffle intro} respects the perverse filtration, in the sense that 
\begin{equation}
\label{eqn:intro compatibility}
\Psi \left(P^d (\CA^T_{Q,W})\right) \subseteq P^d (\CA^T_Q)
\end{equation}
for all $d \geq 1$.

\end{theorem}

\medskip

Thus, \eqref{eqn:intro compatibility} implies that for every quiver $Q$ and potential $W$, formulas \eqref{eqn:intro degree bound} provide necessary conditions for a polynomial to arise from a CoHA element of perverse filtration $\leq d$. This behaves particularly well in the case of tripled quivers with potential, in which for suitable tori, the morphism $\Psi$ is known to be injective and we can conjecturally describe its image (see Section \ref{sec:triple}).  

\medskip 

We also note the work \cite{Kiem}, which describes the intersection cohomology of the coarse moduli spaces $\mathcal{M}_{\bn}(Q)$, and thus provides a description of the BPS Lie algebra. It would be interesting to match this description with the $d=1$ case of Theorem \ref{thm:main}, but the precise connection is not obvious to us.

\medskip

\subsection{Tripled quivers}

Let us discuss a particularly important special case of the framework above. If $Q$ is any quiver, we let $\dQ$ be the double quiver (same vertex set, but we add an arrow $a^* : \oji$ for any arrow $a : \oij$). We also consider the triple quiver $\tilde{Q}$, in which we add a loop $\omega_i$ at every vertex of the quiver. Let $\tilde{W}$ be the canonical cubic potential associated with the quiver $Q$, as recaled in Example \ref{ex:cubic}.  Following the $K$-theoretic statements in \cite[Theorem 2.9, Corollary 2.10]{Zhao} and \cite[Proposition 2.11]{NWheel}, in Proposition \ref{prop:yu} we show that
\begin{equation}
\label{eqn:intro wheel}
\Psi(\CA^T_{\tQ,\tW}) \subseteq \CS 
\end{equation} 
where $T$ is the torus of Example \ref{ex:triple}, and $\CS \subset \CA^T_Q$ denotes the set of polynomials $E(z_{i1},\dots,z_{in_i})_{i \in I}$ which satisfy the wheel conditions \eqref{eqn:wheel general}. It was conjectured in \cite{NGen} that the inclusion \eqref{eqn:intro wheel} becomes an equality if we tensor the two sides of the equation with $\text{Frac}(\BC[\text{Lie }T])$. If we add to this the fact that $\Psi$ is injective for the tripled quiver by \cite[Theorem 10.2]{davison2022integrality} and \cite[Proposition 4.6]{schiffmanncohagenerators}, then Theorem \ref{thm:main} has the following consequence.

\medskip 

\begin{corollary}
\label{cor:intro}

Subject to Conjecture \ref{conj:wheel}, $\Psi$ induces an isomorphism between
$$
P^d (\CA^T_{\tQ,\tW,\mathrm{loc}})
$$
and the set of polynomials $E$ which simulatenously satisfy the degree bounds \eqref{eqn:intro degree bound} and the wheel conditions \eqref{eqn:wheel general}.

\end{corollary}
\medskip

Let $\mathfrak{g}^T_{\Pi_Q}$ be the preprojective BPS Lie algebra defined in \cite{davison2022bps}. It enjoys many interesting properties, such as the fact that its size is given by Kac polynomials (\cite{Mozgovoy}, \cite{davisonkac}), it is isomorphic by \cite{botta2023okounkovs} to the positive half of the Maulik-Okounkov Lie algebra (\cite{moyangian}), and it is known by \cite{DHS} to match the generalized Kac-Moody Lie algebra generated by the intersection cohomology of the good moduli space of representations of the preprojective algebra. As
\[ P^1(\mathcal{A}^T_{\tilde{Q},\tilde{W},\mathrm{loc}}) = \mathfrak{g}^{T}_{\tilde{Q},\tilde{W},\mathrm{loc}} \simeq \mathfrak{g}^{T}_{\Pi_Q,\mathrm{loc}}, \] 
Corollary \ref{cor:intro} gives a precise description of $\mathfrak{g}^{T}_{\Pi_Q,\mathrm{loc}}$ in terms of wheel and limit conditions, and is such a step toward the full understanding of $\mathfrak{g}^T_{\Pi_Q}$.

\medskip

\subsection{Acknowledgements} We would like to thank Ben Davison,  Lucien Hennecart and Sebastian Schlegel Mejia for numerous conversations. We would like to thank Sebastian Schlegel Mejia and Tudor Pǎdurariu for many helpful comments. We gratefully acknowledge the support of the Swiss National Science Foundation grant 10005316.

\bigskip 

\section{Shuffle algebras}
\label{sec:shuffle}

We present certain combinatorial conditions on shuffle algebras, which will be shown to describe the perverse filtration on CoHAs of quivers in the next Section. However, the contents of the present section can be understood independently from the remainder of the paper. 

\medskip 

\subsection{Notations}
\label{sub:notations}

We will consider shuffle algebras in the generality of \cite{NArbitrary}, though as opposed from the trigonometric case studied in \emph{loc. cit.}, we will presently define rational shuffle algebras. The basic input for the construction will consist of:

\medskip 

\begin{itemize}

\item a finite set $I$ (typically the set of vertices of a quiver),

\medskip

\item a $\BQ$-algebra $R$ and a collection of rational functions
\begin{equation}
\label{eqn:def zeta}
\zeta_{ij}(x) \in \frac {R[x]}{x^{\delta_{ij}}}
\end{equation}
for all $i,j \in I$ (in most applications, $\zeta_{ij}$ will be a product of linear factors, one such factor per arrow from $i$ to $j$ in the quiver in the previous bullet).

\end{itemize}

\medskip

\noindent Let us assume that 
\begin{equation}
\label{eqn:zeta degree}
\zeta_{ij}(x) = (-x)^{\#_{ij}} + \Big( \text{lower order terms as } x \sim \infty \Big)
\end{equation}
for various integers $\#_{ij} \geq  - \delta_{ij}$ such that the following \emph{symmetry} property holds
\begin{equation}
\label{eqn:symmetry}
\#_{ij} = \#_{ji}
\end{equation}
for all $i,j \in I$. This allows us to introduce the symmetric bilinear form
$$
\zz \otimes \zz \xrightarrow{( \cdot, \cdot )} \BZ
$$
given by
\begin{equation}
\label{eqn:euler form}
( \bm, \bn ) = - \sum_{i,j \in I} m_i n_j \#_{ij}
\end{equation}
for all $\bm = (m_i)_{i \in I}, \bn = (n_i)_{i \in I}$. We will write
$$
\bs^i = \underbrace{(0,\dots,0,1,0,\dots,0)}_{1 \text{ on }i\text{-th spot}}
$$
for all $i \in I$, and consider the partial order $\bm \leq \bn$ if $m_i \leq n_i$ for all $i \in I$ (thus, $\bm < \bn$ means $\bm \leq \bn$ and $\bm \neq \bn$). We will also use the notation $\b0 = (0,\dots,0)$. In the present paper, we make the convention that the set $\BN$ includes 0.

\medskip

\subsection{Shuffle algebras}
\label{sub:shuffle}

We will now recall the definition of the (rational version of the) Feigin-Odesskii shuffle algebra (\cite{FO}) associated to the data $(I,\zeta_{ij})$ as in the previous Subsection. The relation between this algebra and the cohomological Hall algebras that we will consider in the next Section goes back to the seminal work \cite{KS}. Consider the vector space of polynomials in arbitrarily many variables
\begin{equation}
\label{eqn:big shuffle}
\CV = \bigoplus_{\bn \in \nn} \CV_{\bn}, \quad \text{where} \quad \CV_{(n_i \geq 0)_{i \in I}} = R[z_{i1},\dots,z_{in_i}]^{\text{sym}}_{i \in I} 
\end{equation}
Above, ``sym" refers to polynomials which are color-symmetric, i.e. symmetric in the variables $z_{i1},\dots,z_{in_i}$ for each $i \in I$ separately. The vector space $\CV$ is called the \emph{(big) shuffle algebra} when endowed with the following shuffle product:
\begin{equation}
\label{eqn:mult}
E( z_{i1}, \dots, z_{i n_i})_{i \in I} * E'(z_{i1}, \dots,z_{i n'_i})_{i \in I} = 
\end{equation}
$$
\textrm{Sym} \left[ \frac {E(z_{i1}, \dots, z_{in_i}) E'(z_{i,n_i+1}, \dots, z_{i,n_i+n'_i})}{\bn! \bn'!}
\prod_{i,j \in I} \mathop{\prod_{1 \leq a \leq n_i}}_{n_j < b \leq n_j+n_j'} \zeta_{ij} \left( z_{ia} - z_{jb} \right) \right]
$$
The word ``Sym" in \eqref{eqn:mult} refers to summing over the
\begin{equation*}
(\bn+\bn')! := \prod_{i\in I} (n_i+n'_i)!
\end{equation*}
permutations of the variables $\{z_{i1}, \dots, z_{i,n_i+n'_i}\}$ for each $i$ independently.

\medskip

\begin{example}
\label{ex:non equivariant}

When $\zeta_{ij}(x) = (-x)^{\#_{ij}}$ for all $i,j \in I$, we obtain the non-equivariant cohomological Hall algebra of the quiver with vertex set $I$ and $\#_{ij}+\delta_{ij}$ arrows from $i$ to $j$, for all vertices $i,j \in I$ (\cite[Theorem 2]{KS}). 
    
\end{example}

\medskip 

\begin{example}
\label{ex:triple shuffle}

Suppose $Q$ is a quiver with vertex set $I$ and arrow set $\edge$, and the ground ring $R$ contains distinguished elements $\{\hbar,u_{\alpha}\}_{\alpha \in \edge}$. Consider 
\begin{equation}
\label{eqn:zeta triple}
\zeta_{ij}(x) = 
\left( \frac {x-\hbar}x\right)^{\delta_{ij}} \prod_{\alpha: \oij} (-x-u_\alpha) \prod_{\alpha:\oji}(-x-\hbar+u_\alpha)
\end{equation}
for all $i,j \in I$. As we will recall in Section \ref{sec:triple}, the corresponding (big) shuffle algebra contains the localized preprojective cohomological Hall algebra associated to $Q$, at least under the strong assumption of Subsection \ref{sub:wheel}. 

\end{example}

\medskip

\subsection{Limits}
\label{sub:lim}

We will now introduce certain limit conditions on shuffle algebras, which generalize those considered in \cite[Section 2]{NExts} for the preprojective cohomological Hall algebra associated to the Jordan quiver. 

\medskip

\begin{definition}
\label{def:limit}

For any $d \geq 1$ and $\bn \in \nn$, let $F^d(\CV_{\bn})$ be the set of color-symmetric
$$
E(z_{i1},\dots,z_{in_i})_{i \in I} \in \CV_{\bn}
$$
such that for any partition
\begin{equation}
\label{eqn:partition}
\bn = \bn^1 + \dots + \bn^k
\end{equation}
(for arbitrary $k \geq 1$ and $\bn^1,\dots,\bn^k \in \nn \backslash \b0$), the polynomial
\begin{equation}
\label{eqn:add y}
E(y_1+z_{i,1},\dots,y_1+z_{i,n_i^1}, \dots,y_k+z_{i,n_i-n_i^k+1}, \dots, y_k+z_{i,n_i})_{i \in I}
\end{equation}
has total degree in $y_1,\dots,y_k$ bounded as follows
\begin{equation}
\label{eqn:degree bound}
\emph{deg}_{y_1,\dots,y_k}(E) \leq \frac {d-k}2 - \sum_{1\leq a < b \leq k} (\bn^a,\bn^b)
\end{equation}

\end{definition}

\medskip 

\noindent We will write
$$
F^d (\CV) = \bigoplus_{\bn \in \nn} F^d(\CV_{\bn})
$$
It is clear that
$$
F^1 (\CV) \subset F^2(\CV) \subset \dots \subset \CV 
$$
gives a filtration of $\CV$, i.e. $F^d(\CV) \subset F^{d+1}(\CV)$ and $\bigcup_{d=1}^{\infty} F^d (\CV) = \CV$. 

\medskip

\begin{proposition} 
\label{prop:sub algebra}

We have
$$
F^d (\CV) * F^{d'} (\CV) \subseteq F^{d+d'} (\CV )
$$
for all $d,d' \geq 1$.

\end{proposition} 

\medskip

\begin{proof} Let us suppose that $E$ and $E'$ are shuffle elements (i.e. elements of the big shuffle algebra $\CV$) which satisfy the condition in Definition \ref{def:limit} for positive integers $d$ and $d'$. By the definition of the shuffle product in \eqref{eqn:mult}, $E*E'$ is the sum of 
\begin{equation}
\label{eqn:temp mult}
\frac {E(z_{i1}, \dots, z_{in_i}) E'(z_{i,n_i+1}, \dots, z_{i,n_i+n'_i})}{\bn! \bn'!}
\prod_{i,j \in I} \mathop{\prod_{1 \leq a \leq n_i}}_{n_j < b \leq n_j+n_j'} \zeta_{ij} \left( z_{ia} - z_{jb} \right)
\end{equation}
and its various permutations of variables. To check that $E*E'$ lies in $F^{d+d'}(\CV)$, we must consider all possible partitions of the variable set
$$
\Big\{ z_{i1},\dots,z_{i,n_i+n_i'} \Big\}= S_i^1 \sqcup \dots \sqcup S_i^k 
$$
add $y_a$ to the variables in the set $S_i^a$ (for all $a \in \{1,\dots,k\}$) and then estimate the total degree of the resulting polynomial in the variables $y_1,\dots,y_k$. However, any partition as above is obtained by concatenating partitions
\begin{align} 
&\Big\{ z_{i1},\dots,z_{i,n_i} \Big\}= T_i^1 \sqcup \dots \sqcup T_i^k \label{eqn:part 1} \\
&\Big\{ z_{i,n_i+1},\dots,z_{i,n_i+n_i'} \Big\}= U_i^1 \sqcup \dots \sqcup U_i^k \label{eqn:part 2} 
\end{align}
for various $S_i^a = T_i^a \sqcup U_i^a$. Let us now add $y_a$ to the variables in the sets $T_i^a$ and $U_i^a$ for all $a \in \{1,\dots,k\}$. Because of Definition \ref{def:limit}, we infer that \footnote{While in formula \eqref{eqn:add y} it might seem that we can only add $y_1,\dots,y_k$ to specific variables of $E$, the color-symmetry of $E$ allows us to permute the variables $z_{i1},\dots,z_{in_i}$ arbitrarily.}
\begin{align} 
&E(y_1+z_{i1}, \dots, y_k+z_{in_i}) \quad \text{has degree } \leq \frac {d-\ell}2 - \sum_{1 \leq a < b \leq k} (\boldsymbol{t}^a, \boldsymbol{t}^b) \label{eqn:add 1} \\
&E'(y_1+z_{i,n_i+1}, \dots, y_k+z_{i,n_i+n'_i}) \quad \text{has degree } \leq \frac {d'-\ell'}2 - \sum_{1 \leq a < b \leq k} (\boldsymbol{u}^a, \boldsymbol{u}^b) \label{eqn:add 2}
\end{align} 
where $\boldsymbol{t}^a$, $\boldsymbol{u}^a$ are the $I$-tuples of cardinalities of the sets $T_i^a$, $U_i^a$, respectively, and $\ell$, $\ell'$ denotes the number of $a \in \{1,\dots,k\}$ such that $\sum_{i \in I} |T_i^a| > 0$, $\sum_{i \in I} |U_i^a| > 0$, respectively. We clearly have
\begin{equation}
\label{eqn:easy}
\ell+\ell' \geq k
\end{equation}
Finally, the product of $\zeta$ factors in \eqref{eqn:temp mult} is of the form
\begin{equation}
\label{eqn:contribution 1}
\prod_{1 \leq a,b \leq k} \prod_{i,j \in I} \prod_{z \in T_i^a, z' \in U_j^b} \zeta_{ij}(y_a - y_b + z - z')
\end{equation}
and thus contributes precisely
\begin{equation}
\label{eqn:add 3}
- \sum_{1\leq a \neq b \leq k} (\boldsymbol{t}^a, \boldsymbol{u}^b)
\end{equation}
to the total degree in $y_1,\dots,y_k$. Adding \eqref{eqn:add 1}, \eqref{eqn:add 2}, \eqref{eqn:add 3} together shows that the total degree of \eqref{eqn:temp mult} in $y_1,\dots,y_k$ is at most
\begin{equation}
\label{eqn:amount}
\frac {d+d'-\ell-\ell'}2 - \sum_{1 \leq a < b \leq k} (\boldsymbol{t}^a+\boldsymbol{u}^a, \boldsymbol{t}^b+\boldsymbol{u}^b) 
\end{equation}
(this uses the symmetry of the bilinear form \eqref{eqn:euler form}). Because of the inequality \eqref{eqn:easy}, this is exactly what we needed to prove in order to ensure that $E * E' \in F^{d+d'} (\CV)$.
    
\end{proof}

\begin{proposition} 
\label{prop:sub lie algebra}

Let us define the following super Lie bracket \footnote{In the notation of Subsubsection \ref{subsub:sign twists}, this makes $\CV$ into a superalgebra, once we twist the shuffle product \eqref{eqn:temp mult} by $(-1)^{\psi(\bn,\bn')}$.}
$$
[E,E']^{\chi} = E*E' - (-1)^{(\bn,\bn')} E'*E
$$
for any $E \in \CV_{\bn}$, $E' \in \CV_{\bn'}$. We have
$$
[F^d (\CV) , F^{d'} (\CV)]^{\chi} \subseteq F^{d+d'-1} (\CV )
$$
for all $d,d' \geq 1$. In particular, $F^1(\CV)$ is a super Lie algebra. 

\end{proposition} 

\medskip 

\begin{proof} The proof follows that of Proposition \ref{prop:sub algebra} very closely, and the only distinction is that \eqref{eqn:contribution 1} should be replaced by 
\begin{multline}
\label{eqn:contribution 2}
\prod_{1 \leq a,b \leq k} \prod_{i,j \in I} \prod_{z \in T_i^a, z' \in U_j^b} \zeta_{ij}(y_a - y_b + z - z') - \\ \prod_{1 \leq a,b \leq k} \prod_{i,j \in I} \prod_{z \in T_i^a, z' \in U_j^b} (-1)^{\#_{ij}}\zeta_{ji}(y_b - y_a + z' - z)
\end{multline}
It suffices to focus on the case $\ell+\ell' = k$ in \eqref{eqn:easy}, which means that there are no terms involving $a=b$ in any of the products above. Therefore, the fact that
$$
\zeta_{ij}(x) - (-1)^{\#_{ij}}\zeta_{ij}(-x) \text{ has order }x^{\#_{ij}-1} \text{ as } x \sim \infty
$$
implies that the total degree in $y_1,\dots,y_k$ of \eqref{eqn:contribution 2} is at least one less than the number \eqref{eqn:add 3}, which allows us to replace $d+d'$ by $d+d'-2$ in \eqref{eqn:amount}. 

\end{proof}

\medskip 

\subsection{The key property}
\label{sub:key}

Consider the operation
\begin{equation}
\label{eqn:delta}
\CV \rightarrow \CV, \qquad E \leadsto Eu
\end{equation}
$$
E(z_{i1},\dots,z_{in_i})_{i \in I} u = E(z_{i1},\dots,z_{in_i})_{i \in I} \left(\sum_{i \in I} \sum_{a=1}^{n_i} z_{ia} \right) 
$$
The following Proposition is straightforward, so we leave it to the reader.

\medskip 

\begin{proposition}
\label{prop:delta}

The operation \eqref{eqn:delta} is a derivation of the shuffle product, i.e.
$$
(E*E')u = (Eu) * E' + E * (E'u)
$$
and sends $F^d(\CV)$ to $F^{d+2} (\CV)$ for all $d \geq 1$. 

\end{proposition}

\medskip 

\noindent The following is the key property of the filtration of Definition \ref{def:limit}.

\medskip 

\begin{proposition}
\label{prop:key}

For any $d \geq 1$ and $\bn \in \nn$, we have
\begin{equation}
\label{eqn:recursion 1}
E(z_{i1},\dots,z_{in_i})_{i \in I} \in F^d(\CV_{\bn} )
\end{equation}
if and only if
\begin{equation} \tag{A1}
\label{eqn:recursion 2}
\deg_y \left( E(y+z_{i1},\dots,y+z_{in_i})_{i \in I} \right) \leq \frac {d-1}2
\end{equation}
and
\begin{equation} \tag{A2}
\label{eqn:recursion 3}
E(z_{i1},\dots,z_{im_i}|z_{i,m_i+1},\dots,z_{in_i})_{i \in I} \in \sum_{d'+d'' \leq d - 2(\bm,\bn-\bm)} F^{d'}(\CV_{\bm}) \otimes F^{d''} (\CV_{\bn-\bm} )
\end{equation}
for all $\b0 < \bm < \bn$. The meaning of the notation in equation \eqref{eqn:recursion 3} is that we express $E$ as a sum of polynomials $E'(z_{i1},\dots,z_{im_i})_{i \in I} E''(z_{i,m_i+1},\dots,z_{in_i})_{i \in I}$ and require the corresponding sum of tensors $E' \otimes E''$ to lie in the specific filtration pieces mentioned in the right-hand side.

\end{proposition}

\medskip 

\noindent Note that Proposition \ref{prop:key} determines the filtration $F^d\CV_{\bn}$ uniquely. Indeed, \eqref{eqn:recursion 2} implies that
\begin{equation}
\label{eqn:simple}
F^d \CV_{\bs^i} = \text{span} \left\{z_{i1}^0, z_{i1}^1,\dots, z_{i1}^{\left \lfloor \frac {d-1}2 \right \rfloor} \right\}
\end{equation}
and then we use the equivalence of \eqref{eqn:recursion 1} and \eqref{eqn:recursion 2}-\eqref{eqn:recursion 3} to recursively (in $\bn$) determine $F^d \CV_{\bn}$.

\medskip 

\begin{proof} Let us begin with the implication $\Rightarrow$, by assuming $E \in F^d\CV_{\bn}$. Property \eqref{eqn:recursion 2} is the particular condition in Definition \ref{def:limit} for $k=1$, so it holds. In order to check property \eqref{eqn:recursion 3}, let us consider the operation
$$
F^d \CV_{\bn} \ni E \leadsto \Big(\text{sum of tensors } E' \otimes E''\Big) 
$$
described in the statement of the Theorem. Let us consider arbitrary partitions
\begin{align*} 
&\Big\{ z_{i1},\dots,z_{im_i} \Big\}= T_i^1 \sqcup \dots \sqcup T_i^k  \\
&\Big\{ z_{i,m_i+1},\dots,z_{in_i} \Big\}= U_i^1 \sqcup \dots \sqcup U_i^{\ell}  
\end{align*}
of the variable sets of $E'$ and $E''$, respectively. We must prove that when we add $y_a$ to the variables in $\cup_{i \in I} T_i^a$ for all $a \in \{1,\dots,k\}$ (respectively add $x_b$ to the variables in $\cup_{i \in I} U_i^b$ for all $b \in \{1,\dots,\ell\}$) the total degree of the sum of tensors $E' \otimes E''$ in $y_1,\dots,y_k, x_1,\dots,x_{\ell}$ is at most
$$
\frac {d-2(\bm,\bn-\bm)-k-\ell}2 - \sum_{1 \leq a < b \leq k} (\bt^a,\bt^b) - \sum_{1 \leq a < b \leq \ell} (\bu^a,\bu^b)
$$
However, this follows immediately from the fact that $E \in F^d\CV_{\bn}$, since we can invoke the property of Definition \ref{def:limit} for the partition
$$
\Big\{ z_{i1},\dots,z_{i,m_i}, z_{i,m_i+1},\dots,z_{i,n_i} \Big\}= T_i^1 \sqcup \dots \sqcup T_i^k \sqcup U_i^1 \sqcup \dots \sqcup U_i^{\ell}  
$$ 
For the implication $\Leftarrow$ in the statement of Proposition \ref{prop:key}, let us consider expression \eqref{eqn:add y} as a function of $y_1,\dots,y_k$ for any partition \eqref{eqn:partition}. If $k = 1$, then \eqref{eqn:recursion 2} implies the bound on the degree in $y_1$ that we need to check. If $k > 1$, then \eqref{eqn:recursion 3} for $\bm = \bn^1$ implies that the degree of \eqref{eqn:add y} in $y_1,\dots,y_k$ is
$$
\leq \frac {d'-1}2 + \frac {d''-k+1}2 - \sum_{2\leq a < b \leq k} (\bn^a,\bn^b) \leq \frac {d-k}2 - \sum_{1 \leq a < b \leq k} (\bn^a,\bn^b)
$$
which is exactly what we need to prove in order to check the condition in Definition \ref{def:limit} for $E$. \end{proof}

\bigskip 

\section{Perverse Filtrations}
\label{sec:coha}

We will now relate the combinatorial notions of the previous Section with cohomological Hall algebras (CoHAs for short) of quivers with potential. We first recall these algebras in complete generality and then prove Theorem \ref{thm:main}.

\medskip

\subsection{Stacks of quiver representations}
\label{sub:quiver}

Let $Q$ be any quiver, with vertex set $I$ and arrow set $\edge$. For any $\alpha \in \edge$, we will write $s(\alpha)$ and $t(\alpha)$ for the source and target of $\edge$, respectively. Given any dimension vector $\bn \in \nn$, define 
\begin{align*}
\mathrm{Rep}_{\bn}(Q) &:= \prod_{\alpha \in \edge} \mathrm{Hom}(\BC^{\bn_{s(\alpha)}},\BC^{\bn_{t(\alpha)}}) \\ 
\GL_{\bn} &:= \prod_{i \in I} \GL_{\bn_i} \end{align*} 
where the group $\GL_{\bn}$ acts on $\mathrm{Rep}_{\bn}(Q)$ by conjugation. Let 
\[\mathfrak{M}_{\bn}(Q) := \mathrm{Rep}_{\bn}(Q)/\GL_{\bn} \] 
be the moduli stack of $\bn$ dimensional representations of $Q$. Let 
\[\mathcal{M}_{\bn}(Q) := \Spec \left(\BC[\mathrm{Rep}_{\bn}(Q)]^{\GL_{\bn}} \right) \] 
be the coarse moduli space parameterizing semisimple points of the aforementioned stack. We then have the affinization map given by taking the direct sum of successive quotients in the Jordan-H\"older filtration of a quiver representation
\[ \JH_{\bn} \colon \mathfrak{M}_{\bn}(Q) \rightarrow \mathcal{M}_{\bn}(Q).\]
Suppose we have an algebraic torus $T$ endowed with characters
\begin{equation}
\label{eqn:elementary character}
e^{u_\alpha} : T \rightarrow \BC^*
\end{equation}
for all $\alpha \in \edge$. We can think of the exponent as an equivariant parameter
\begin{equation}
\label{eqn:equivariant parameter}
u_\alpha \in \HHT := \HH(\BT) = \BC[\text{Lie } T]
\end{equation}
The torus $T$ acts on the space of quiver representations by scaling the arrow $\alpha$ with the character $e^{u_{\alpha}}$. We then define the quotient stacks
\[ \mathfrak{M}_{\bn}^{T}(Q) := \mathrm{Rep}_{\bn}(Q)/(\GL_{\bn} \times T) \] 
and
\[ \mathcal{M}^{T}_{\bn}(Q) := \mathcal{M}_{\bn}(Q)/T\] 
which are well-defined since the $T$ and $\GL_{\bn}$ actions on $\mathrm{Rep}_{\bn}(Q)$ commute. We have an induced morphism
\[ \JH^{T}_{\bn} \colon \mathfrak{M}^{T}_{\bn}(Q) \rightarrow \mathcal{M}^{T}_{\bn}(Q). \] 
Given a quiver $Q=(I,\Omega)$, we define the opposite quiver as $Q^{\op}=(I,\Omega^{\op})$ with the same vertex set but the arrows reversed and equivariant parameters $-u_{\alpha}$. 
\medskip 

\subsection{Potential} 
\label{sub:potential}

A \emph{potential} $W$ on $Q$ is any
$\BC$-linear combination of oriented cycles in $Q$ modulo cyclic permutations. If $W$ is a single cyclic word, we define for any $\alpha \in \edge$
\[ \frac{\partial W}{\partial \alpha} = \sum_{W = c\alpha c'} c'c \] 
and we extend this definition linearly to any $W$. The \emph{Jacobi algebra} of $Q$ is
$$
\text{Jac}(Q,W) = \BC[W] \Big / \left( \frac{\partial W}{\partial \alpha} \right)_{\alpha \in \edge}
$$
Any potential $W$ induces a regular function 
\[ \Tr_{\bn}(W): \mathrm{Rep}_{\bn}(Q) \rightarrow \BC 
\] 
which is $\GL_{\bn}$ invariant due to the conjugation invariance of the trace. Therefore, we obtain an induced regular function on the moduli stack 
\[ \Tr_{\bn}(W): \mathfrak{M}_{\bn}(Q) \rightarrow \BC 
\] 
Whenever it will be clear from context, we will also use the notation $\Tr_{\bn}(W): \mathcal{M}_{\bn}(Q) \rightarrow \BC$ to denote the induced regular function on the coarse moduli space. If the potential is $T$ invariant (i.e. the sum of the torus weights of the arrows in every constituent cycle of $W$ is 0), then we obtain an induced regular function on the stack
$$
\mathrm{Tr}^T_{\bn}(W): \mathfrak{M}^T_{\bn}(Q) \rightarrow \mathbb{C}
$$ 
and similarly on the coarse moduli space $\mathcal{M}^{T}_{\bn}(Q)$.

\medskip

\subsection{Vanishing cycles and constructible sheaves} In this section, we introduce vanishing cycle sheaves with the goal of setting up the notation for later on. For details on the construction, we refer to \cite{achar}. 

\medskip

Let $X$ be a finite type separated scheme over $\BC$. We let $\mathcal{D}_{\mathrm{c}}(X)$ be the derived category of constructible complexes of $\BQ$ vector spaces over $X$. We let $\mathcal{D}^{\mathrm{b}}_{\mathrm{c}}(X)$ be the full subcategory of bounded complexes and $\mathcal{D}^+_{\mathrm{c}}(X)$ the subcategory of complexes that are bounded below. We let $\Per(X)\subset \mathcal{D}_{\mathrm{c}}(X)$ be the category of perverse sheaves. It is an abelian category realized as the heart of the perverse $t$-structure $(\pD_{\mathrm{c}}^{\leq 0}(X),\pD_{\mathrm{c}}^{\geq 0}(X))$. Let $\ptau^{\leq d}, \ptau^{\geq d}$ denote the perverse truncation functors and 
$$
\pH^d=\ptau^{\leq d}\ptau^{\geq d}[d]\colon \mathcal{D}_{\mathrm{c}}(X)\rightarrow\Per(X)
$$
the perverse cohomology functors. If $\mathcal{F}$ is a complex of constructible sheaves on $X$, we let 
$$
\pH(\mathcal{F})\coloneqq \bigoplus_{d\in\BZ}\pH^d(\mathcal{F})[-d]
$$
be its total perverse cohomology. Now let $T$ be any torus acting on the scheme $X$, with quotient stack $X/T$. The shifted perverse $t$-structure on the $T$-equivariant derived category $\mathcal{D}^{+}_{\mathrm{c}}(X/T)$ is defined by setting ${}^{\mathfrak{p}}\!\tau^{\leq d} : = {}^{\mathfrak{p}}\!\tau^{\leq d - \dim(T)}$ and ${}^{\mathfrak{p}}\!\tau^{\geq d} := {}^{\mathfrak{p}}\!\tau^{\geq d - \dim(T)}  $. Let $\Per(X/T)$ be the heart with respect to this shifted perverse structure. Explicitly, it is defined so that any constructible sheaf $\mathcal{F}$ on $X/T$ is perverse if and only if its pullback $(X \rightarrow X/T)^*\mathcal{F}$ to $X$ is perverse. In particular, if $Y \subset X$ is a smooth connected component, then the shifted constant sheaf $\BQ^{\vir}_{Y/T}:= \BQ_{Y/T}[\dim Y]$ is in  $\Per(X/T)$. Note that this is slight abuse of notation since usually, perverse sheaves on Artin stacks (See \cite{perversestacks}) are defined by descent from schemes so that if $\mathfrak{M}$ is a smooth stack, $\BQ_{\mathfrak{M}}[\dim \mathfrak{M}]$ is perverse. Hence, our definition of perverse sheaves on $X/T$ differs by a shift by the dimension of torus. 

\medskip

Suppose that $X$ is a smooth algebraic variety with an action of $T$, and let $f\colon X\rightarrow \mathbb{C}$ be a $T$ invariant regular function, then there is a perverse exact functor 
$$
\varphi_f\colon \mathcal{D}^+_{\mathrm{c}}(X/T)\rightarrow \mathcal{D}^+_{\mathrm{c}}(X/T)
$$
called the vanishing cycle functor. We refer to \cite[Section 8.6]{schapira} for the definition and \cite[Proposition 2.13]{davison2020cohomological} for the essential properties that it satisfies. \newline

\subsection{Cohomological Hall algebras}
\label{sub:coha}

We henceforth restrict to \emph{symmetric} quivers, i.e. those which have as many arrows from $i$ to $j$ as from $j$ to $i$, for any vertices $i,j$. Consider any potential $W$ on $Q$ which is invariant with respect to a torus $T$ as in Subsection \ref{sub:quiver}. Let use denote by $\varphi_{\Tr(W)}$ the vanishing cycle functor defined for the regular function $\mathrm{Tr}^T_{\bn}(W): \mathfrak{M}_{\bn}^T(Q) \rightarrow \mathbb{C}$. Then we define 
\begin{equation}
\label{eqn:coha definition}
\mathcal{A}^T_{Q,W} := \bigoplus_{\bn \in \mathbb{N}^{I}} \HH(\mathfrak{M}^T_{\bn}(Q), \varphi_{\Tr_{\bn}(W)} \BQ^{\vir})
\end{equation}
where $\BQ^{\vir}$ is the trivial perverse sheaf $\mathbb{Q}_{\mathfrak{M}^T_{\bn}(Q)}[\dim(\mathfrak{M}_{\bn})] = \mathbb{Q}_{\mathfrak{M}^T_{\bn}(Q)}[-(\bn,\bn)]$ on the smooth stack $\mathfrak{M}^T_{\bn}(Q)$ (recall the bilinear form \eqref{eqn:euler form}). The object \eqref{eqn:coha definition} will be denoted by $\CA_{Q,W}$ if the torus is trivial and $\CA^T_Q$ if the potential is 0.

\medskip 

\subsubsection{Algebra structure}In \cite{KS}, a graded associative algebra structure was constructed on $\mathcal{A}_{Q,W}$. We outline the construction. The product is defined via the following correspondence. Given two dimension vectors $\bn, \bn^{\prime} \in \nn$, let $\mathfrak{M}^{T}_{\bn,\bn^{\prime}}(Q)$ be the stack of exact sequences $0 \rightarrow \rho_1 \rightarrow \rho_2 \rightarrow \rho_3 \rightarrow 0$ where $\rho_1$ and $\rho_3$ are $T$-equivariant representations of dimensions $\bn$ and $\bn^{\prime}$ respectively. Clearly the representation $\rho_2$ is of dimension $\bn+\bn^{\prime}$. Then we consider the following correspondence diagram
\begin{equation}
\label{eqn:diagram}
\begin{tikzcd}
	& {\mathfrak{M}_{\bn,\bn^{\prime}}^{T}(Q)} & \\
	{\mathfrak{M}^{T}_{\bn}(Q) \times_{\BT} \mathfrak{M}^{T}_{\bn^{\prime}}(Q)} && {\mathfrak{M}^{T}_{\bn+\bn^{\prime}}(Q)}
	\arrow["{{q_1 \times q_3}}"', from=1-2, to=2-1]
	\arrow["{{q_2}}", from=1-2, to=2-3]
\end{tikzcd}
\end{equation}
The morphism $q_1 \times q_3$ is an affine fibration and hence induces an isomorphism
\begin{align} \label{q1q_2}
(q_1 \times q_3)^{*}& \colon \HH(\mathfrak{M}^{T}_{\bn}(Q)   \times_{\BT} \mathfrak{M}^{T}_{\bn^{\prime}}(Q),  \varphi_{\Tr_{\bn}(W) \boxplus \Tr_{\bn^{\prime}}(W)} \BQ^{\vir}) \rightarrow \nonumber \\ & \HH(\mathfrak{M}^{T}_{\bn, \bn^{\prime}}(Q), \varphi_{\Tr(W)|_{\mathfrak{M}^{T}_{\bn,\bn^{\prime}}(Q)}}\BQ)[(\bn,\bn)+(\bn^{\prime},\bn^{\prime})]  
\end{align}

The morphism $q_2$ is proper and thus induces a morphism 
\begin{align}(q_2)_{*} & \colon \HH(\mathfrak{M}^{T}_{\bn,\bn^{\prime}}(Q),  \varphi_{\Tr(W)|_{\mathfrak{M}^{T}_{\bn,\bn^{\prime}}(Q)}}\BQ)[(\bn,\bn)+(\bn^{\prime},\bn^{\prime})] \rightarrow \nonumber \\ &\HH(\mathfrak{M}^{T}_{\bn+\bn^{\prime}}(Q),\varphi_{\Tr(W)}\BQ^{\vir}). 
\end{align} 
We also have the Thom-Sebastiani isomorphism
\begin{align}
&\mathrm{TS}^{\prime}: \HH(\mathfrak{M}^{T}_{\bn}(Q)   \times_{\BT} \mathfrak{M}^{T}_{\bn^{\prime}}(Q),  \varphi_{\Tr_{\bn}(W) \boxplus \Tr_{\bn^{\prime}}(W)} \BQ^{\vir}) \nonumber \\ & \simeq \HH(\mathfrak{M}^{T}_{\bn}(Q)   \times_{\BT} \mathfrak{M}^{T}_{\bn^{\prime}}(Q),  \varphi_{\Tr_{\bn}(W)}\BQ^{\vir} \boxplus \varphi_{\Tr_{\bn^{\prime}}(W)} \BQ^{\vir})
\end{align}
We would like to apply the K\"unneth isomorphism. However, since the product is taken over $\BT$, it doesn't apply directly. We work with the following assumption.

\medskip

\begin{assumption} \label{assumption1}
$Q,W,T$ are chosen so that $\mathcal{A}^{T}_{Q,W}$ is a free module over $\HHT = \HH(\BT).$
This assumption is known to hold when $\mathcal{A}_{Q,W}$ is cohomologically pure, as explained in \cite[Section 3.6] {DW} and \cite[A.1.12]{zhu2016introductionaffinegrassmanniansgeometric}. The aforementioned purity is known when $W=0$ and when $Q,W$ is the tripled quiver with canonical cubic potential (see Section \ref{sec:triple}) by \cite{davison2022integrality}. 

\end{assumption}

\medskip 

\noindent The preceding assumption implies that we have an isomorphism of graded vector spaces 
\begin{align*}
&\mathrm{Ku} \colon \HH(\mathfrak{M}^{T}_{\bn}(Q)   \times_{\BT} \mathfrak{M}^{T}_{\bn^{\prime}}(Q),  \varphi_{\Tr_{\bn}(W)}\BQ^{\vir} \boxplus \varphi_{\Tr_{\bn^{\prime}}(W)} \BQ^{\vir}) \\ & \simeq  \HH(\mathfrak{M}^{T}_{\bn}(Q), \varphi_{\Tr_{\bn}(W)}\BQ^{\vir}) \otimes_{\HHT} \HH(\mathfrak{M}^{T}_{\bn^{\prime}}(Q), \varphi_{\Tr_{\bn^{\prime}}(W)}\BQ^{\vir})
\end{align*}
Since we will be working under Assumption \ref{assumption1}, we will denote the composition $\mathrm{Ku} \circ \mathrm{TS}^{\prime}$ by $\mathrm{TS}$. Then we can define the morphism 
\begin{equation}
\label{eqn:coha 0}
m_{\bn,\bn'} : \mathcal{A}^{T}_{Q,W,\bn} \otimes_{\HHT} \mathcal{A}^{T}_{Q,W,\bn^{\prime}} \xrightarrow{(q_1 \times q_3)^{*} \circ (q_2)_{*} \circ \mathrm{TS}^{-1}} \mathcal{A}^{T}_{Q,W,\bn+\bn^{\prime}}
\end{equation}
It was proved in \cite[Section 2.3]{KS} that taking the direct sum of \eqref{eqn:coha 0} over all $\bn,\bn^{\prime} \in \nn$ yields an associative $\HHT$-linear algebra structure on $\mathcal{A}^{T}_{Q,W}$ which is called the \emph{cohomological Hall algebra of the quiver with potential} (CoHA for short). 

\medskip

\begin{example}
\label{ex:any}

For $W=0$, it was proved in \cite[Theorem 2]{KS} that $\mathcal{A}_Q$ is isomorphic to the shuffle algebra $\mathcal{V}$ defined in Example \ref{ex:non equivariant}. If there is a torus $T$ endowed with characters \eqref{eqn:elementary character}, then the cohomological Hall algebra $\mathcal{A}^T_Q$ is isomorphic to the shuffle algebra $\mathcal{V}$ with respect to the shuffle kernel \footnote{This is not proved in \cite[Theorem 2]{KS}, however the same proof works in the presence of equivariant structure. See also \cite[Section 17]{J} and \cite[Section 10]{davison2022integrality}. } 
\begin{equation}
\zeta_{ij}(x) = \frac{\prod_{\alpha: \oij} (-x+u_\alpha)}{(-x)^{\delta_{ij}}}.
\end{equation}
We will write $F^d(\mathcal{A}^T_Q)$ for the pullback of the filtration $F^d(\CV)$ (of Definition \ref{def:limit}) under the isomorphism $\mathcal{A}^T_{Q} \cong \mathcal{V}$.
\end{example}

\medskip

\begin{example}
\label{ex:triple}

For any quiver $Q$, we can consider the double quiver $\dQ$ with the same vertex set, but twice larger arrow set: for every arrow $\alpha : \oij$ we add an arrow $\alpha^* : \oji$. Moreover, we also consider the triple quiver $\tQ$ where we add a loop $\omega_i$ at every vertex $i \in I$. In this setup, the torus acting is $T = \BC^* \times (\BC^*)^{|\edge|}$, with characters corresponding to the following sequences of integers
\begin{align*}
u_{\omega_i} &\leadsto -1 \times (0,\dots,0)\\ u_\alpha &\leadsto 0 \times \underbrace{(0,\dots,0,-1,0,\dots,0)}_{-1 \text{ on }\alpha\text{-th spot}} \\ 
u_{\alpha^{*}} &\leadsto 1 \times \underbrace{(0,\dots,0,1,0,\dots,0)}_{1 \text{ on }\alpha\text{-th spot}}\end{align*}
If we write $u_{\omega_i} = -\hbar$ for all $i$, then we have the equation $u_{\alpha^*} = \hbar - u_\alpha$. Then by example \ref{ex:any}, it follows that $\mathcal{A}^T_{\tilde{Q}}$ is isomorphic to $\mathcal{V}$ with shuffle kernel 
$$
\zeta_{ij}(x) = \left( \frac {x-\hbar}{x} \right)^{\delta_{ij}} \prod_{\alpha: \oij} (-x-u_\alpha) \prod_{\alpha:\oji}(-x-\hbar+u_\alpha) 
$$
Note that this shuffle algebra is isomorphic to $\mathcal{V}$ of Example \ref{ex:triple shuffle}. From now on, we shall identify 
$$
\CA^T_{\tQ} \cong \CV
$$
endowed with the shuffle product of Example \ref{ex:triple shuffle}.

\end{example}

\medskip 

\begin{example}
\label{ex:cubic}

With $\tQ$ as in the previous example, consider the \emph{canonical cubic potential}
$$
\tW = \sum_{\alpha \in \edge} \Big( \omega_{t(\alpha)} \alpha\alpha^* - \omega_{s(\alpha)}\alpha^*\alpha \Big) 
$$
In \cite[Appendix]{Ren:2015zua} and \cite{Yang_2019}, it was shown that 
$$
\CA^T_{\tQ,\tW} \cong \CA^T_{\Pi_Q}
$$
where the object in the right-hand side denotes the preprojective CoHA of the quiver $Q$, defined for the Jordan quiver in \cite{schiffmann2012cherednik} and for arbitary quivers in \cite{Yang_2018}. 

\end{example}

\medskip 

\subsubsection{Module structure}  \label{modulestructure}For any $\bn \in \mathbb{N}^{I}$, let 
$$
\mathcal{S}^{T}_{Q,\bn} = \mathrm{H}(\mathfrak{M}^{T}_{\bn}(Q),\mathbb{Q})
$$
(although $\mathcal{S}^{T}_{Q,\bn} = \mathcal{A}^{T}_{Q,\bn}$, we use different notation to emphasize the different roles that will be played by the objects $\CS$ and $\CA$ in \eqref{eqn:action}). Let 
\begin{equation}
\label{eqn:coproduct}
\Delta: \mathfrak{M}_{\bn}(Q)  \rightarrow \mathfrak{M}_{\bn}(Q) \times_{\BT} \mathfrak{M}_{\bn}(Q)
\end{equation}
be the diagonal embedding. Since the composition 
$$
\mathfrak{M}^{T}_{\bn}(Q)  \xrightarrow{\Delta} \mathfrak{M}^{T}_{\bn}(Q) \times_{\BT} \mathfrak{M}^{T}_{\bn}(Q) \xrightarrow{0\boxplus \Tr^T_{\bn}(W) } \mathbb{C}
$$
is the same function as $\Tr_{\bn}^T(W)$, we have that the pullback and the Thom-Sebastiani isomorphism yield a morphism of sheaves 
\[ \mathbb{Q}_{\mathfrak{M}^{T}_{\bn}(Q)} \boxtimes \varphi_{\Tr(W)}\mathbb{Q}^{\vir}_{{\mathfrak{M}^{T}_{\bn}(Q)}}  \rightarrow \Delta_{*}(\varphi_{\Tr(W)}\BQ^{\vir}_{\mathfrak{M}^{T}_{\bn}(Q)}). \] 
After taking global sections, we obtain a morphism
\begin{equation}
\label{eqn:action}
\mathcal{S}^T_{Q,\bn} \otimes_{\HHT} \mathcal{A}^{T}_{Q,W,\bn} \rightarrow \mathcal{A}^{T}_{Q,W,\bn}
\end{equation}
which coincides with the cup poduct when $W=0$, since then the vanishing cycle is the trivial sheaf and the cup product is defined to be the pullback along the diagonal morphism. It is proved in \cite[Section 2.6]{davisonlocalized} that \eqref{eqn:action} makes $\mathcal{A}^{T}_{Q,W,\bn}$ into a $\mathcal{S}^{T}_{Q,\bn}$-module, with the ring structure on $\mathcal{S}^T_{Q,\bn}$ given by cup product. 

\medskip 

\subsubsection{Coalgebra structure} We will now define the \textit{localized} coalgebra structure on the cohomological Hall algebra (we refer to \cite{davisonlocalized} and \cite{JKL} for more details). The algebra structure on the CoHA is given by pulling back along the morphism $q_1 \times q_3$ and then pushing forward along the morphism $q_2$ of \eqref{eqn:diagram}. To define the coproduct, we use the same idea, but in reverse: we pull back along the morphism $q_2$ and then push forward along the morphism $q_1 \times q_3$ to get a map in opposite direction. The only issue with this is that the morphism $q_1 \times q_3$ is not proper, so the procedure above is only well defined up to localization by certain Euler classes. To make this precise, recall the torus $T$ which acts on quiver representations with respect to characters $e^{u_\alpha}$ as in \eqref{eqn:elementary character}. Let $\GL^T_{\bn} = \GL_{\bn} \times T$ and define
\begin{equation} \label{eqn:eulerclass}
\mathrm{eu}^{T}_{\bn^{1},\bn^{2}}(\edge) :=  \prod_{\alpha \in \edge} \prod_{\substack{1 \leq a \leq n^{1}_{s(\alpha)} \\ 1 \leq b \leq n^{2}_{t(\alpha)}}} (z^{\prime \prime}_{t(\alpha)b}-z^{\prime}_{s(\alpha)a}+u_{\alpha})\in \HH^{\GL^T_{\bn^{1}}} \otimes_{\HH^T} \HH^{\GL^T_{\bn^{2}}}
\end{equation}
and 
\[ \mathrm{eu}^{T}_{\bn^{1},\bn^{2}}(I) := \prod_{i \in I} \prod_{\substack{1 \leq a \leq n^{1}_{i} \\ 1 \leq b \leq n^{2}_{i}}} (z^{\prime \prime}_{ib}-z^{\prime}_{ia}) \in \HH^{\GL^T_{\bn^{1}}} \otimes_{\HH_{T}} \HH^{\GL^T_{\bn^{2 }}}.  \] 
Above, the notation $z^{\prime}_{*}$ and $z^{\prime \prime}_{*}$ refers to the generators of the (maximal torus) equivariant cohomology of a point in the first and second tensor factors of the expression above, respectively. Given two $\nn$ graded $\HHT$ modules $V$ and $W$, we define the localized tensor product \[ V_{\bn^{1}} \tilde{\otimes} W_{\bn^{2}} := V_{\bn^{1}} \otimes_{\HHT} W_{\bn^{2}}[\mathrm{eu}^{T}_{\bn^{1},\bn^{2}}(\edge)^{-1} \mathrm{eu}^{T}_{\bn^{1},\bn^{2}}(\edge^{\mathrm{op}})^{-1}] \] and correspondingly 
\begin{equation}
\label{eqn:localized tensor product}
V \tilde{\otimes} W := \bigoplus_{\bn^{1},\bn^{2} \in \nn} (V_{\bn^{1}} \tilde{\otimes} W_{\bn^{2}})
\end{equation}
More generally, if we are given graded $\nn$ modules $V^1,\cdots,V^k$ and $1 \leq p < q \leq k$, then we define 
\[ (V^1_{\bn^1} \otimes_{\HHT} \cdots \otimes_{\HHT} V^k_{\bn^k})_{(p,q)}  : = (V^1_{\bn^1} \otimes_{\HHT} \cdots \otimes_{\HHT} V^k_{\bn^k})[\mathrm{eu}^{T,(p,q)}_{\bn^{p},\bn^{q},}(\edge)^{-1} \mathrm{eu}^{T,(p,q)}_{\bn^{p},\bn^{q}}(\edge^{\mathrm{op}})^{-1}] \] where $\mathrm{eu}^{T,(p,q)}_{\bn^{p},\bn^{q},}(\edge)$ is defined by equation \eqref{eqn:eulerclass}, except that we see variables $z^{\prime \prime}$ in $q$-th tensor and $z^{\prime}$ in $p$th tensor. Similarly for any number of pairs $(p_i,q_i)$ with $1 \leq p_i < q_i \leq k$, one can define $(V^1_{\bn^1} \otimes_{\HHT} \cdots \otimes V^k_{\bn^k})_{(p_1,q_1),\cdots,(p_r,q_r)}$ by succsessive localizations. Finally one takes the direct sum along all $\bn^i$ to define 
$$
(V^1 \otimes_{\HHT} \cdots  \otimes_{\HHT}  V^k)_{(p_1,q_1),\cdots,(p_r,q_r)}
$$
We now define the localized coproduct. Let $\bn \in \nn$ and let $\bn^{1}, \bn^{2} \in \nn$ such that $\bn^{1}+\bn^{2}=\bn$. Recall the correspondence diagram \eqref{eqn:diagram}. We consider the pullback along the morphsim $q_2$, giving a morphism of graded vector spaces: 
\begin{align*}
(q_2)^{*}  \colon \HH(\mathfrak{M}^{T}_{\bn}(Q),\varphi_{\Tr(W)}\BQ^{\vir})\rightarrow\HH(\mathfrak{M}^{T}_{\bn^1,\bn^{2}}(Q),  \varphi_{\Tr(W)|_{\mathfrak{M}^{T}_{\bn^1,\bn^{2}}(Q)}}\BQ)[-(\bn,\bn)]
\end{align*}
and the morphism $((q_1 \times q_3)^{*})^{-1}$, which is the inverse of isomorphism $q_1 \times q_3$ defined in \eqref{q1q_2}. For any $\bn^{1}+\bn^{2}=\bn$, we have the morphism
\[ \mathrm{TS} \circ \left((q_1 \times q_3)^* \right)^{-1} (q_2)^{*} \colon \mathcal{A}^{T}_{Q,W,\bn} \rightarrow \mathcal{A}^{T}_{Q,W,\bn^{1}} \otimes_{\HHT} \mathcal{A}^{T}_{Q,W,\bn^{2}}  \]
Finally, we invert the Euler class along the $0$ section of the affine fibration $q_1 \times q_3$. More precisely, we define 
\[ \Delta_{\bn^{1},\bn^{2}}:=\frac{\mathrm{eu}^{T}_{\bn^{1},\bn^{2}}(I)}{\mathrm{eu}^{T}_{\bn^{1},\bn^{2}}(\edge)}\mathrm{TS} \circ \left((q_1 \times q_3)^* \right)^{-1} (q_2)^{*} \colon \mathcal{A}^{T}_{Q,W,\bn} \rightarrow \mathcal{A}^{T}_{Q,W,\bn^{1}} \tilde{\otimes} \mathcal{A}^{T}_{Q,W,\bn^{2}}  \]
Taking the direct sum along all $\bn^1+\bn^2 = \bn$ defines the \textit{localized coproduct} \[\Delta: \mathcal{A}^{T}_{Q,W} \rightarrow  \mathcal{A}^{T}_{Q,W} \tilde{\otimes}  \mathcal{A}^{T}_{Q,W} \]
Note that localization with Euler classes does not make the localized coproduct trivial. In fact, it was proved in \cite[Proposition 4.1]{davisonlocalized} that the identity morphism $\mathcal{A}^T_{Q,W} \otimes_{\HHT} \mathcal{A}^T_{Q,W} \rightarrow \mathcal{A}^T_{Q,W} \tilde{\otimes} \mathcal{A}^T_{Q,W}$ is injective.  \newline

\subsubsection{Bialgebra structure} The algebra and (localized) coalgebra structure are compatible, in the sense that they satisfy the following compatibility. In what follows, we let $m$ denote the CoHA product. The following is proved in \cite[Proposition 5.9]{davisonlocalized} (see also \cite[Proposition 5.5]{botta2023okounkovs} for the statement in the $T$-equivariant setting).

\begin{proposition}\label{prop: coha_loc_coprod}
    The following diagram commutes 
    \begin{equation}
\begin{tikzcd}
	{\mathcal{A}^{T}_{Q,W} \otimes_{\HH^T}  \mathcal{A}^{T}_{Q,W}} & {\mathcal{A}^{T}_{Q,W}} \\
	{(\mathcal{A}^{T}_{Q,W} \otimes_{\HH^T}  \mathcal{A}^{T}_{Q,W}\otimes_{\HH^T}  \mathcal{A}^{T}_{Q,W}\otimes_{\HH^T}  \mathcal{A}^{T}_{Q,W})_{(12),(23),(34),(14)}} \\
	{(\mathcal{A}^{T}_{Q,W} \otimes_{\HH^T}  \mathcal{A}^{T}_{Q,W}\otimes_{\HH^T}  \mathcal{A}^{T}_{Q,W}\otimes_{\HH^T}  \mathcal{A}^{T}_{Q,W})_{(13),(23),(14),(24)}} \\
	{(\mathcal{A}^{T}_{Q,W} \otimes_{\HH^T}  \mathcal{A}^{T}_{Q,W}) \widetilde{\otimes} (\mathcal{A}^{T}_{Q,W} \otimes_{\HH^T}  \mathcal{A}^{T}_{Q,W})} & {\mathcal{A}^{T}_{Q,W} \tilde{\otimes} \mathcal{A}^{T}_{Q,W}}
	\arrow["m", from=1-1, to=1-2]
	\arrow["{a \circ(\Delta \otimes \Delta)}"', from=1-1, to=2-1]
	\arrow["\Delta", from=1-2, to=4-2]
	\arrow["{\id \otimes \tilde{\mathrm{sw}} \otimes \id}"', from=2-1, to=3-1]
	\arrow["{=}"', from=3-1, to=4-1]
	\arrow["{m \otimes m}"', from=4-1, to=4-2]
\end{tikzcd}
    \end{equation}
with subscripts $(ij)$ denoting localization with respect to the $i$-th and $j$-th factors,
\begin{multline*} 
a \colon (\mathcal{A}^{T}_{Q,W} \otimes_{\HH^T}  \mathcal{A}^{T}_{Q,W}\otimes_{\HH^T}  \mathcal{A}^{T}_{Q,W}\otimes_{\HH^T}  \mathcal{A}^{T}_{Q,W})_{(12),(34)} \to  \\
(\mathcal{A}^{T}_{Q,W} \otimes_{\HH^T}  \mathcal{A}^{T}_{Q,W}\otimes_{\HH^T}  \mathcal{A}^{T}_{Q,W}\otimes_{\HH^T}  \mathcal{A}^{T}_{Q,W})_{(12),(23),(34),(14)}
\end{multline*}
being the localisation morphism. Moreover, $\tilde{\mathrm{sw}}$ is the following twisting morphism for the localised tensor product
\begin{equation}
\tilde{\mathrm{sw}}_{\bn^{1},\bn^{2}} = \emph{sw} \circ (-1)^{\tau(\bn^{1}, \bn^{2})} \mathrm{eu}^{T}_{\bn^{1},\bn^{2}}(\edge^{\mathrm{op}})^{-1} \mathrm{eu}^{T}_{\bn^{1 },\bn^{2}}(\edge) 
\end{equation} 
for any two dimension vectors $\bn^{1},\bn^{2}$. Here $\mathrm{sw}$ is the usual swap morphism with the Kozul-sign given by the cohomological grading, i.e 
\[ \mathrm{sw}(\gamma_1 \otimes \gamma_2) = (-1)^{|\gamma_1||\gamma_2|} \gamma_2 \otimes \gamma_1\] where for any $\gamma \in \mathcal{A}^T_{Q,W}$, $|\gamma|$ is parity of cohomological degree and $\tau$ is the sign twist, given by 
\begin{align*}
\tau \colon \nn \times \nn & \rightarrow \BZ/2\BZ  \\ (\bn^{1},\bn^{2})  & \mapsto (\bn^{1},\bn^{1})(\bn^{2},\bn^{2}) + (\bn^{1},\bn^{2}). 
\end{align*} 

\end{proposition} 

\subsubsection{Sign twists}
\label{subsub:sign twists}

Note that the morphism $\tilde{\mathrm{sw}}_{\bn^{1},\bn^{2}}$ is composed of two non-trivial morphisms. One is the cup product with the Euler classes and the other one  is multiplication by $(-1)^{\tau(\bn^{1},\bn^{2})}$. It is possible to get rid of this sign at cost of introducing a twisted sign in the algebra structure defining $\mathcal{A}^{T}_{Q,W}$. Given two dimension vectors $\bn^{1}, \bn^{2}$, we consider any bilinear form $\psi$ on $(\mathbb{Z}/2\mathbb{Z})^I$ so that we have \[ \psi(\bn^{1},\bn^{2}) + \psi(\bn^{2},\bn^{1}) \equiv \tau(\bn^{1},\bn^{2}) \mod 2. \] Note that such a bilinear form always exists, since $\tau(\bn,\bn) \equiv 0 \mod 2 $. Let
$$
\mathcal{A}^{T,\psi}_{Q,W} = \mathcal{A}^{T}_{Q,W}
$$
as a vector space, but with the multiplication twisted according to the formula 
$$
m^{\psi}_{\bn^{1},\bn^{2}} = (-1)^{\psi(\bn^{1},\bn^{2})} m_{\bn^{1},\bn^{2}}.
$$
With the sign twist, the algebra $\mathcal{A}^{T,\psi}_{Q,W}$ satisfies exactly the same commutative diagram as in Proposition \ref{prop: coha_loc_coprod}, but with $\tilde{\mathrm{sw}}$ therein replaced by \[\tilde{\mathrm{sw}}^{\psi}:= \mathrm{sw} \circ  \mathrm{eu}^{T}_{\bn^{1},\bn^{2}}(\edge^{\mathrm{op}})^{-1} \mathrm{eu}^{T}_{\bn^{1 },\bn^{2}}(\edge) \]

\begin{example}
When $W=0$, if $\gamma \in \mathcal{A}^T_{Q,\bn}$ then $|\gamma|= (\bn,\bn)$ (this is because $\mathcal{A}^T_{Q,\bn} = \HH(\mathfrak{M}^T_{\bn}(Q),\mathbb{Q}[-(\bn,\bn)])$ and $\HH(\mathfrak{M}^T_{\bn}(Q),\mathbb{Q})$ has even dimensional cohomology). Thus if $E \in \mathcal{A}^{T}_{Q,\bn}$ and $E^{\prime} \in \mathcal{A}^T_{Q,\bn^{\prime}}$ then the super Lie bracket \begin{align*}
[E,E^{\prime}] &:= E *^{\psi} E^{\prime} - (-1)^{|E||E^{\prime}|} E^{\prime} *^{\psi} E \\ &= (-1)^{\psi(\bn,\bn^{\prime})} E * E^{\prime} - (-1)^{\psi(\bn^{\prime},\bn)+ (\bn,\bn)(\bn^{\prime},\bn^{\prime})} E^{\prime}*E \\ &= (-1)^{\psi(\bn,\bn^{\prime})} (E * E^{\prime}- (-1)^{(\bn,\bn^{\prime})} E^{\prime}*E) 
\end{align*} 
By Proposition \ref{prop:sub lie algebra}, $F^1(\mathcal{V})$ forms a super Lie algebra in $\mathcal{A}^{T,\psi}_{Q}$ for the zeta function in Example \ref{ex:any}. 
\end{example}

\medskip

\subsection{Perverse filtration and the decomposition theorem}
\label{sub:perverse}

The semi simplification morphism $\JH^{T}_{\bn}$ is almost never proper. However, it is proved in \cite[Lemma 4.1] {davison2020cohomological} (see \cite{kinjo2024decompositiontheoremgoodmoduli} for a more general statement) that we can still apply the decomposition theorem. Specifically, the direct image $(\JH^{T}_{\bn})_{*} \varphi_{\Tr(W)}\BQ^{\vir}$ splits, i.e. we have a non-canonical isomorphism in $\mathcal{D}^{+}_c(\mathcal{M}^{T}_{\bn}(Q))$: 
\begin{equation} \label{splitting1}
(\JH^{T}_{\bn})_{*} \varphi_{\Tr(W)} \BQ^{\vir} \simeq \bigoplus_{d \in \mathbb{Z}} {}^{\mathfrak{p}} \! \mathcal{H}^{d}((\JH^{T}_{\bn})_{*} \varphi_{\Tr(W)}\BQ^{\vir})[-d]. 
\end{equation} 
Furthermore, it is proved in \cite{davison2020cohomological} and \cite{DW} (in the $T$-equivariant case) that \[{}^{\mathfrak{p}}\!\mathcal{H}^{d}((\JH^{T}_{\bn})_{*} \varphi_{\Tr(W)}\BQ^{\vir})=0 \] for all $d <1$, so the above splitting starts from $d \geq 1$. For any $d$, applying the natural transformation ${}^{\mathfrak{p}}\!(\tau^{\leq d}) \rightarrow \id$, gives a split morphism \begin{equation} \label{truncation} {}^{\mathfrak{p}}\!\tau^{\leq d}((\JH^{T}_{\bn})_{*} \varphi_{\Tr(W)} \BQ^{\vir} )\rightarrow (\JH^{T}_{\bn})_{*} \varphi_{\Tr(W)} \BQ^{\vir}
\end{equation}
Taking global sections defines the perverse filtration. 
\begin{definition}
    The perverse filtration  \[0 \subset P^{1}(\mathcal{A}^{T}_{Q,W,\bn}) \subset P^{2}(\mathcal{A}^{T}_{Q,W,\bn}) \subset \cdots \subset \mathcal{A}^{T}_{Q,W,\bn} \] of $\mathcal{A}^{T}_{Q,W,\bn}$ is defined as \[ P^{d}( \mathcal{A}^{T}_{Q,W,\bn}) := \mathrm{H}( \mathcal{M}_{\bn}^{T}(Q), {}^{\mathfrak{p}}\!\tau^{\leq d}((\JH^{T}_{\bn})_{*} \varphi_{\Tr(W)} \BQ^{\mathrm{vir}}) ) \hookrightarrow \mathcal{A}^{T}_{Q,W,\bn}. \]
for all $d \geq 1$.
\end{definition}

\medskip  

Note that the filtration is increasing since the morphism \eqref{truncation} splits. 
It is proven in \cite{davison2020cohomological} that the perverse filtration is compatible with the CoHA multiplication. \newline 

We now describe the associated graded CoHA with respect to the perverse filtration:
\begin{equation}
\label{eqn:graded coha}
\text{Gr}_P (\CA_{Q,W}^{T,\psi} ) = \bigoplus_{d=1}^{\infty} P^d \left(\CA_{Q,W}^{T} \right) / P^{d-1}\left(\CA_{Q,W}^{T} \right)
\end{equation}
Since the perverse filtration doesn't depend on the multiplication, we henceforth interpret it as a filtration on the sign twisted CoHA $\mathcal{A}^{T,\psi}_{Q,W}$. 

\medskip

\begin{definition}
The BPS Lie algebra $\fg^{T}_{Q,W}: = \bigoplus_{\bn \in \mathbb{N}^{I} \backslash {0}} \fg^{T}_{Q,W,\bn}$ is defined as \[  \fg^{T}_{Q,W,\bn} :=  P^1(\mathcal{A}^{T}_{Q,W,\bn})\] where the Lie bracket is induced from the associative algebra structure on $\mathcal{A}^{T,\psi}_{Q,W}$. The BPS sheaves are defined as 
\[ \mathcal{BPS}^{T}_{Q,W,\bn}  =  {}^{\mathfrak{p}} \! \mathcal{H}^{1}((\JH^{T}_{\bn})_{*} \varphi_{\Tr(W)}\BQ^{\vir}) \in \Per(\mathcal{M}^T_{\bn}(Q)) \]  
\end{definition}

\medskip

Thus by definition \[P^1(\mathcal{A}^{T}_{Q,W,\bn}) = \HH(\mathcal{M}^{T}_{\bn}(Q), \mathcal{BPS}^{T}_{Q,W,\bn}[-1]) \]
Note that it is not obvious why the subspace $\fg^{T}_{Q,W} \subset \mathcal{A}^{T,\psi}_{Q,W}$ is preserved under the Lie bracket. This is proved in \cite[Corollary 6.11]{davison2020cohomological} and is in fact a consequence of the PBW theorem that we now turn to.

\medskip 

\subsection{The PBW theorem}

In Section \ref{modulestructure}, we defined an action of the CoHA
$$
\mathcal{S}^T_Q \cong \bigoplus_{\bn \in \nn} \HHT[z_{i1},\dots,z_{in_i}]_{i \in I}^{\sym}
$$
on $\mathcal{A}^{T}_{Q,W}$ and thus on $\mathcal{A}^{T,\psi}_{Q,W}.$ Consider the subring $\mathbb{C}[u] \hookrightarrow \mathcal{S}^T_Q$ where $u \mapsto \sum_{i \in I} \sum_{a=1}^{n_i} z_{ia}$. Let 
\begin{equation}
\label{eqn:bps into a}
\fg^T_{Q,W}[u] = \mathbb{C}[u] \cdot \fg^{T}_{Q,W} \hookrightarrow \mathcal{A}^{T,\psi}_{Q,W}
\end{equation}
The following theorem is proved in \cite[Theorem C]{davison2020cohomological}, and is adapted to the $T$-equivariant case in \cite[Theorem 3.4]{DW} and \cite[Corollary 2.3] {hennecart2026degenerationscohas2calabiyaucategories}.

\begin{proposition}\label{integrality theorem}
    There is an isomorphism of graded $\HHT$-algebras 
    \[ \mathrm{Sym}_{\HHT}(\fg^{T}_{Q,W}[u]) \simeq \mathrm{Gr}_{P}(\mathcal{A}^{T,\psi}_{Q,W})\]
\end{proposition}

\medskip 

\noindent Since $\mathrm{Gr}_{P}(\mathcal{A}^{T,\psi}_{Q,W}) \simeq \mathcal{A}^T_{Q,W}$ as graded $\HHT$-modules, Proposition \ref{integrality theorem} implies that the morphism of graded $\HHT$ modules
\[\Sym_{\HHT}(\fg^T_{Q,W}[u]) \rightarrow \mathcal{A}^T_{Q,W} \] defined by extending \eqref{eqn:bps into a} multiplicatively is an isomorphism, which is commonly referred to as the cohomological integrality theorem. 

\medskip

The perverse filtration is, by definition, respected by the action of the graded ring $\HHT$. Furthermore, in \cite[Lemma 5.8]{davison2020cohomological} it is proved that 

\begin{proposition}\label{prop: eulerclassesonperversefiltration}
Let $V$ be any $T$-equivariant vector bundle on $\mathfrak{M}_{\bn}(Q)$, and let $\mathrm{eu}^T(V)$ be its equivariant Euler class. Then
\[ \mathrm{eu}^T(V) \cdot P^d(\mathcal{A}^{T}_{Q,W,\bn}) \subset P^{d+ 2 \mathrm{rank}(V)}(\mathcal{A}^{T}_{Q,W,\bn}).\]
\end{proposition}

\medskip 

\subsection{The automorphism action} \label{defn:act}

Consider the morphism 
\begin{align*}
\act:  \BCu \times \mathfrak{M}^{T}_{\bn}(Q) \rightarrow \mathfrak{M}^{T}_{\bn}(Q)
\end{align*} 
defined by sending the tautological line bundle $\mathcal{L}$ on $\BCu$ and any family of representations $\mathcal{F} \in \mathfrak{M}^{T}_{\bn}(Q)$ to $\mathcal{F \otimes L} \in \mathfrak{M}^{T}_{\bn}(Q)$. Pulling back and taking global sections give a map of cohomologically graded vector spaces
\[ \act^*_y: \mathrm{H}(\mathfrak{M}^{T}_{\bn}(Q),\varphi_{\Tr(W_{\bn})}\mathbb{Q}^{\vir}_{\mathfrak{M}^T_{\bn}(Q)})  \rightarrow \mathrm{H}( \BCu,\BQ) \otimes \mathrm{H}(\mathfrak{M}^{T}_{\bn}(Q),\varphi_{\Tr(W_{\bn})}\mathbb{Q}^{\vir}_{\mathfrak{M}^T_{\bn}(Q)})
\] where the variable $y$ means that we are identifying $\HH(\BCu)$ with $\mathbb{C}[y]$. 
In particular, for any $\gamma \in \mathcal{A}^{T}_{Q,W}$ we have \[\act^{*}_y(\gamma) = \sum_{n \geq 0} y^n \otimes \gamma_{(n)} \] for suitable $\gamma_{(n)} \in \mathcal{A}^{T}_{Q,W}$. This allows us to define the linear map \[\partial \colon \mathcal{A}^{T}_{Q,W} \rightarrow \mathcal{A}^{T}_{Q,W}\] by $\gamma \mapsto \gamma_{(1)}$. For example, when $W=0$, the linear map $\partial$ is explicitly given by extracting the linear term in $y$ from formula \eqref{eqn:explicit}. Let $\tau: \mathrm{B}\mathbb{C}^* \otimes \mathrm{B}\mathbb{C}^* \rightarrow \mathrm{B}\mathbb{C}^*$ defined by $(\mathcal{L}_1,\mathcal{L}_2) \mapsto \mathcal{L}_1 \otimes \mathcal{L}_2$. Then the following diagram commutes 
\[\begin{tikzcd}
	{\mathrm{B}\mathbb{C}^* \times \mathrm{B}\mathbb{C}^*\times  \mathfrak{M}^{T}_{\bn}(Q)} & {\mathrm{B}\mathbb{C}^* \times \mathfrak{M}^T_{\bn}(Q)} \\
	{\mathrm{B}\mathbb{C}^* \times \mathfrak{M}^T_{\bn}(Q)} & {\mathfrak{M}^T_{\bn}(Q)}
	\arrow["{\id \times \mathrm{act}}", from=1-1, to=1-2]
	\arrow["{\tau \otimes \mathrm{id}}"', from=1-1, to=2-1]
	\arrow["{\mathrm{act}}", from=1-2, to=2-2]
	\arrow["{\mathrm{act}}"', from=2-1, to=2-2]
\end{tikzcd}\]
Thus, if we identify the cohomology of the first $\mathrm{B}\mathbb{C}^*$ with $\mathbb{C}[y]$ and that of the second $\mathrm{B}\mathbb{C}^*$ with $\mathbb{C}[z]$, then it follows by pullback that \[ (1 \otimes \mathrm{act}^*_z )(\mathrm{act}^*_y) = \mathrm{act}^*_{y \otimes1+1 \otimes z} \]
Comparing coefficients, we infer that 
\[ \act^{*}_y(\gamma) = \mathrm{exp}^{y \partial }(\gamma) = \gamma + y \partial(\gamma) + \frac{y^2 \partial^2(\gamma)}{2!}+ \cdots \]
Note that the sum above is finite since the operator $\partial$ decreases cohomological degree by $2$ and $\mathcal{A}^T_{Q,W}$ is bounded below in the cohomological degree. In fact, this morphism also interacts well with the perverse filtration, as in the following result (\cite[Proposition 10.2]{J} and \cite[Proposition 4.5]{DW}). 

\medskip 

\begin{proposition}\label{partialuder}
The action of $\partial$ decreases the perverse degree by $2$, i.e we have 
\begin{align*}
\partial\left( P^{d}(\mathcal{A}^{T}_{Q,W}) \right) & \subset P^{d-2}(\mathcal{A}^{T}_{Q,W})
\end{align*} 
for all $d$. In particular, if $\gamma \in P^d(\mathcal{A}^T_{Q,W})$ then 
\begin{equation}
\label{eqn:equation action}
\emph{deg}_y \left( \mathrm{act}^{*}(\gamma) \right) \leq \frac {d-1}2
\end{equation}
Moreover, if
$$
\gamma = \beta u^k 
$$
for $\beta \in \fg^{T}_{Q,W,\bn}$, then $\partial(\gamma) = k|\bn| \beta u^{k-1}$ where $|\bn|=\sum_{i \in I } n_i$. 

\end{proposition}

\medskip 

\noindent We note the following useful property, which will allow us to compute $\partial$ when $W=0$. 

\medskip

\begin{proposition} \label{prop: actionwhenWiszero}
For $W=0$, we have 
\begin{equation}
\label{eqn:explicit}
\mathrm{act}^{*}(f) = f(y+z_{i1},\cdots,y+z_{i n_i})
\end{equation}
for any color-symmetric polynomial $f(z_{i1},\cdots,z_{i n_i})_{i \in I} \in \mathcal{A}^{T}_{Q,\bn} \simeq \HHT[z_{i1},\dots,z_{in_i}]_{i \in I}^{\emph{sym}}$.
\end{proposition}

\begin{proof}
The variables $z_{ia}$ are the Chern roots of the tautological bundle $\mathcal{V}_{i}$ on the vertex $i$. We have $\mathrm{act}^{*}(\mathcal{V}_i) = \mathcal{V}_i \boxtimes \mathcal{L}$, where $\mathcal{L}$ is the tautological line bundle on $\BCu$. Thus it follows that $\mathrm{act}^{*}(z_{ia}) = z_{ia}+y$, which implies the statement. 
\end{proof}

\medskip

\subsection{The associated graded algebra}

Let \[ \Delta: \mathcal{A}^{T}_{Q,W} \rightarrow \mathcal{A}^{T}_{Q,W} \tilde{\otimes} \mathcal{A}^{T}_{Q,W} \] be the localized coproduct. We define a filtration on $\mathcal{A}^{T}_{Q,W} \tilde{\otimes} \mathcal{A}^{T}_{Q,W}$ as follows
\[P^{d}(\mathcal{A}^{T}_{Q,W,\bn^{1}} \tilde{\otimes} \mathcal{A}^{T}_{Q,W,\bn^{2}}) := \sum_{\ell \geq 0} (\mathrm{eu}^{T}_{\bn^{1},\bn^{2}}(\edge) \mathrm{eu}^{T}_{\bn^{1},\bn^{2}}(\Omega^{\mathrm{op}}))^{-\ell} P^{d+2\ell D}(\mathcal{A}^{T}_{Q,W,\bn^{1}} \otimes \mathcal{A}^{T}_{Q,W,\bn^{2}}) \]
for any $d$ and $\bn^1,\bn^2 \in \nn$, where $D = \mathrm{deg}(\mathrm{eu}^{T}_{\bn^{1},\bn^{2}}(\edge) )$. Above, we set \[ P^{d}(\mathcal{A}^{T}_{Q,W,\bn^{1}} \otimes \mathcal{A}^{T}_{Q,W,\bn^{2}})\sum_{d^1+d^2=d} P^{d^1}(\mathcal{A}^{T}_{Q,W,\bn^{1}}) \otimes P^{d^2}(\mathcal{A}^{T}_{Q,W,\bn^{2}}) \] for all $d$. 

\medskip 

\begin{proposition} \label{prop:coproductpreserve}
    The localized coproduct $\Delta$ preserves the perverse filtrations. 
\end{proposition}

\medskip

\begin{proof} Let $s: \mathfrak{M}^T_{\bn^{1}}(Q) \times_{\BT} \mathfrak{M}_{\bn^{2}}(Q)  \rightarrow \mathfrak{M}^T_{\bn^{1},\bn^{2}}(Q)$ be the map given by taking a pair of representations to the trivial exact sequence. Notice that $(q_1 \times q_3) \circ s=\mathrm{Id}$, so $s^{*} \circ (q_1 \times q_3)^{*} = \mathrm{Id}$ and hence $s^{*} = ((q_1 \times q_3)^{*})^{-1}.$ Thus, we have 
$$
((q_1 \times q_3)^{*})^{-1} \circ q_2^{*} = \oplus^{*}
$$
where $\oplus$ is the direct sum morphism \[ \oplus \colon  \mathfrak{M}_{\bn^{1}}(Q) \times_{\BT} \mathfrak{M}_{\bn^{2}}(Q) \rightarrow \mathfrak{M}_{\bn^{1}+\bn^{2}}(Q).\] Note that the following diagram commutes 
\begin{equation}
\label{eqn:diagram 0}
\begin{tikzcd}
	{\mathfrak{M}_{\bn^{1}}(Q) \times_{\BT} \mathfrak{M}_{\bn^{2}}(Q)  } & {\mathfrak{M}_{\bn^1+\bn^2}(Q)} \\
	{\mathcal{M}_{\bn^{1}}(Q) \times_{\BT} \mathcal{M}_{\bn^{2}}(Q) } & {\mathcal{M}_{\bn^1+\bn^2}(Q)}
	\arrow["\oplus", from=1-1, to=1-2]
	\arrow["{\JH^T_{\bn^1} \times_{\BT} \JH^T_{\bn^2} }"', from=1-1, to=2-1]
	\arrow["\JH^T_{\bn^1+\bn^2}", from=1-2, to=2-2]
	\arrow["{\oplus_{\mathcal{M}} }"', from=2-1, to=2-2]
\end{tikzcd}
\end{equation}
Consider the map of constructible sheaves in $\mathcal{D}(\mathcal{M}_{\bn^1+\bn^2}(Q))$ given by taking pullback along $\oplus_{\mathcal{M}}$: 
\[ (\JH_{\bn^1+\bn^2})_{*} \varphi_{\Tr(W)}\mathbb{Q}^{\vir}_{\mathfrak{M}^T_{\bn^1+\bn^2}(Q)} \rightarrow (\oplus_{\mathcal{M}})_{*} (\oplus_{\mathcal{M}})^{*}(\JH_{\bn^1+\bn^2})_{*} \varphi_{\Tr(W)}\mathbb{Q}^{\vir}_{\mathfrak{M}^T_{\bn^1+\bn^2}(Q)}\]
Since the diagram \eqref{eqn:diagram 0} is Cartesian, applying Thom-Sebastiani isomorphism, induces a morphism of sheaves 
\begin{align*} (\JH_{\bn^1+\bn^2})_{*}& \varphi_{\Tr(W)}\mathbb{Q}^{\vir}_{\mathfrak{M}_{\bn^1+\bn^2}(Q)} \rightarrow  \\ &(\oplus_{\mathcal{M}})_{*}( (\JH_{\bn^1})_{*} \varphi_{\Tr(W)}\mathbb{Q}^{\vir}_{\mathfrak{M}_{\bn^{1}}(Q)} \boxtimes (\JH_{\bn^2})_{*} \varphi_{\Tr(W)}\mathbb{Q}^{\vir}_{\mathfrak{M}_{\bn^{2}}(Q)})[-2(\bn^{1},\bn^{2})]  \end{align*}
Taking the perverse truncation induces a morphism of sheaves 
\begin{align}{}^{\mathfrak{p}}&\!\tau^{\leq d}(\JH_{\bn^1+\bn^2})_{*} \varphi_{\Tr(W)}\mathbb{Q}^{\vir}_{\mathfrak{M}_{\bn^1+\bn^2}(Q)}  \rightarrow \label{eqn:directsumpullbacksheaves}\\ & {}^{\mathfrak{p}}\!\tau^{\leq d-2(\bn^{1},\bn^{2})}(\oplus_{\mathcal{M}})_{*}( (\JH_{\bn^1})_{*} \varphi_{\Tr(W)}\mathbb{Q}^{\vir}_{\mathfrak{M}_{\bn^{1}}(Q)} \boxtimes (\JH_{\bn^2})_{*} \varphi_{\Tr(W)}\mathbb{Q}^{\vir}_{\mathfrak{M}_{\bn^{2}}(Q)})  \nonumber \end{align}
Since the morphism $\oplus_{\mathcal{M}}$ is finite by \cite[Lemma 2.1]{svenrein}, we can commute past perverse truncation functor and $(\oplus_{\mathcal{M}})_{*}$. Thus, taking global sections implies that 
\begin{equation} \label{eqn:perversedirectsum}
\oplus^{*}: P^d(\mathcal{A}^{T}_{Q,W,\bn}) \rightarrow P^{d-2(\bn^{1},\bn^{2})} (\mathcal{A}^{T}_{Q,W,\bn^{1}} \otimes \mathcal{A}^{T}_{Q,W,\bn^{2}} ) 
\end{equation}  
By Proposition \ref{prop: eulerclassesonperversefiltration}, it follows that the cup product with $\frac {\mathrm{eu}^{T}_{\bn^{1},\bn^{2}}(I)}{\mathrm{eu}^{T}_{\bn^{1},\bn^{2}}(\edge)}$ induces \[ P^{d-2(\bn^{1},\bn^{2})}(\mathcal{A}^{T}_{Q,W,\bn^{1}} \otimes  \mathcal{A}^{T}_{Q,W,\bn^{2}}) \rightarrow P^{d}(\mathcal{A}^{T}_{Q,W,\bn^{1}} \tilde{\otimes} \mathcal{A}^{T}_{Q,W,\bn^{2}})\] 
Since $\Delta_{\bn^{1},\bn^{2}}= \frac {\mathrm{eu}^{T}_{\bn^{1},\bn^{2}}(I)}{\mathrm{eu}^{T}_{\bn^{1},\bn^{2}}(\edge)} \cup \oplus^{*}$, we are done. 

\end{proof}

\subsection{Primitive elements}

Proposition \ref{prop:coproductpreserve} shows that we have an induced morphism of associated graded algebras (with respect to the perverse filtration)
\[ \mathrm{Gr}(\Delta): \mathrm{Gr}(\mathcal{A}^{T}_{Q,W}) \rightarrow \mathrm{Gr}(\mathcal{A}^{T}_{Q,W} \tilde{\otimes} \mathcal{A}^{T}_{Q,W}) 
\]
Note that this morphism is non-zero; in fact, it was proved in \cite[Lemma 6.3]{davison2020cohomological} that the natural morphism \[\mathrm{Gr}(\mathcal{A}^{T}_{Q,W}) \otimes \mathrm{Gr}(\mathcal{A}^{T}_{Q,W})  \rightarrow \mathrm{Gr}(\mathcal{A}^{T}_{Q,W} \tilde{\otimes} \mathcal{A}^{T}_{Q,W}) \]  is injective. Let $\mathcal{P} \subset \mathrm{Gr}(\mathcal{A}^{T}_{Q,W})$ be the set of primitive elements, i.e $\mathrm{Gr}(\Delta)(x) = x \otimes 1 + 1 \otimes x$. The space $\mathcal{P}$ was calculated for trivial $T$ in \cite[Proposition 6.7]{davison2020cohomological}. The similar argument works for general $T$, and we have the following.

\medskip 

\begin{proposition} \label{prop:primitivecalculation}
Let $\fg^{T}_{Q,W}[u]  \subset \mathrm{Gr}(\mathcal{A}^{T}_{Q,W})$ be the subspace defined by the image of the action of $\mathbb{C}[u]$ defined in \eqref{eqn:bps into a} on $\fg^{T}_{Q,W} \subset \mathrm{Gr}(\mathcal{A}^{T}_{Q,W})$. The aforementioned subspace is precisely the set of primitive elements, i.e. 
$$
\mathcal{P} = \fg^{T}_{Q,W}[u]
$$
\end{proposition}

\medskip

\begin{proof}We first show that the space $\fg^{T}_{Q,W}$ consists of primitive elements. We show that if $\bn^1,\bn^2 \neq \b0$, then $\oplus^{*}_{\bn^{1},\bn^{2}}(\fg^{T}_{Q,W,\bn})=0$ in the associated graded algebra, where $\bn=\bn^1+\bn^2$. We claim that the morphism \begin{align}{}^{\mathfrak{p}} \! \mathcal{H}^{1}&(\JH_{*} \varphi_{\Tr(W)}\mathbb{Q}^{\vir}_{\mathfrak{M}^T_{\bn}(Q)})  \rightarrow \label{eqn:degenratedirectsum} \\ & {}^{\mathfrak{p}} \! \mathcal{H}^{1-2(\bn^{1},\bn^{2})}\!(\oplus_{\mathcal{M}})_{*}( \JH_{*} \varphi_{\Tr(W)}\mathbb{Q}^{\vir}_{\mathfrak{M}^T_{\bn^{1}}(Q)} \boxtimes_{\oplus} \JH_{*} \varphi_{\Tr(W)}\mathbb{Q}^{\vir}_{\mathfrak{M}^T_{\bn^{2}}(Q)})\nonumber   \end{align} obtained by applying the perverse cohomology functor to the morphism of sheaves (\ref{eqn:directsumpullbacksheaves}) is zero. This will show the claim after taking the derived global section. 

\medskip 
    
To do this, we first show that the morphism \eqref{eqn:degenratedirectsum} is zero when $W=0$. If the morphism is zero when $W=0$ then applying the functor $\varphi_{\Tr(W)}$ proves the claim, since the functor $\varphi_{\mathrm{Tr}(W)}$ is perverse t-exact and it commutes with the pushforward $\JH_{*}$ by \cite[Lemma 4.1]{davison2020cohomological}. When $W=0$, \[ {}^{\mathfrak{p}} \! \mathcal{H}^{1}(\JH_{*}\mathbb{Q}^{\vir}_{\mathfrak{M}^T_{\bn}(Q)})= \mathcal{BPS}^T_{Q,\bn} = \mathcal{IC}_{\mathcal{M}^T_{\bn}(Q)}\] by \cite[Theorem A]{davison2020cohomological} and \cite[Theorem 4.6]{svenrein}, where $\mathcal{IC}_{\mathcal{M}^T_{\bn}(Q)} \in \Per(\mathcal{M}^T_{\bn}(Q))$ is a semisimple perverse sheaf, called the intersection cohomology sheaf. Thus in the equation (\ref{eqn:degenratedirectsum}) we have a morphism of semisimple perverse sheaves with distinct support when both $\bn^{1}$ and $\bn^{2}$ are non-zero. 

\medskip 

Since the coproduct is linear with respect to action of $u$, it follows that if $\alpha$ is primitive then $\Delta(u^n \cdot \alpha) = u^n \alpha \otimes 1 + 1 \otimes u^n \alpha$. Therefore, the subspace 
$$
\fg^{T}_{Q,W} \otimes \HH(\BCu,\mathbb{Q})
$$
is primitive, and so $\fg^{T}_{Q,W}[u] \subseteq \mathcal{P}$. We now show that in fact, these are all the primitive elements. We claim that the algebra generated by primitive elements $\mathcal{P}$ is an ordinary bialgebra with respect to $\Delta$, i.e. one doesn't need to localize the tensor product as in \eqref{eqn:localized tensor product}. This follows from the observation that the twisting morphism can be written as 
\[ \mathrm{sw} \circ (-1)^{\tau(\bn^{1}, \bn^{2})} \prod_{\alpha \in \edge} \prod_{\substack{1 \leq a \leq n^{1}_{s(\alpha)} \\ 1 \leq b \leq n^{2 }_{t(\alpha)}}}\left( 1+  \frac{2u_\alpha}{z^{\prime \prime}_{t(\alpha)b}-z^{\prime}_{s(\alpha)a}-u_\alpha} \right) \]
Since the $\HH^T$ action respects the preserve filtration, and dividing by Euler classes decreases the perverse degree by Propostion \ref{prop: eulerclassesonperversefiltration}, it follows that the twisting morphism is not localized in the associated graded algebra and thus the algebra generated by primitive elements is a connected bialgebra, which is commutative in $\mathrm{Gr}(\mathcal{A}^{T,\psi}_{Q,W})$. Thus by Milnor-Moore theorem, we have injection \[ \Sym(\fg^{T}_{Q,W}[u]) \subseteq \Sym(\mathcal{P}) \subseteq \mathcal{A}^{T}_{Q,W}.\] However by the cohomological integrality theorem (Proposition \ref{integrality theorem}), the inclusion $\Sym(\fg^T_{Q,W}[u]) \hookrightarrow \mathcal{A}^{T}_{Q,W}$ is an isomorphism and hence $\mathcal{P} = \fg^{T}_{Q,W}[u]$. 
    
\end{proof}

\noindent We have following analogue of Proposition \ref{prop:key} for the perverse filtration on $\mathcal{A}^{T}_{Q,W}$.

\medskip

\begin{proposition} \label{prop:key2}
For any $d \geq 1$ and $\bn \in \mathbb{N}^I$, we have 
\begin{equation} \label{eqn:propertyCoHA1}\gamma \in P^d(\mathcal{A}^{T}_{Q,W,\bn}) \end{equation}
if and only if 
\begin{equation} \tag{B1} \label{eqn:propertyCoHA2} 
\emph{deg}_y \left( \mathrm{act}^*(\gamma)  \right) \leq \frac {d-1}2 
\end{equation}  
and \begin{equation} \tag{B2}\label{eqn: propertyCoHA3}
\oplus^{*}_{\bm,\bn-\bm}(\gamma)\subseteq \sum_{d^{\prime}+d^{\prime \prime} \leq d - 2(\bm,\bn-\bm)} P^{d^{\prime}}(\mathcal{A}^{T}_{Q,W,\bm}) \otimes P^{d^{\prime \prime}}( \mathcal{A}^{T}_{Q,W,\bn-\bm})  \end{equation} for all $\b0 < \bm <\bn$. 
\end{proposition}

\medskip

\begin{proof} If $\gamma \in P^d(\mathcal{A}^{T}_{Q,W,\bn})$, then \eqref{eqn:propertyCoHA2} is proved in \eqref{eqn:equation action}. Similarly, \eqref{eqn: propertyCoHA3} is the content of Proposition \ref{prop:coproductpreserve}. Let us now prove that \eqref{eqn:propertyCoHA2} and \eqref{eqn: propertyCoHA3} imply that $\gamma \in P^d(\mathcal{A}^{T}_{Q,W,\bn})$. Suppose that 
$$
\gamma \in P^{i}(\mathcal{A}^{T}_{Q,W,\bn}) 
$$
for some $i>d$. Then since  \[\oplus^{*}_{\bm,\bn-\bm}(\gamma)\subseteq \sum_{d^{\prime}+d^{\prime \prime} \leq d - 2(\bm,\bn-\bm)} P^{d^{\prime}}(\mathcal{A}^{T}_{Q,W,\bm}) \otimes P^{d^{\prime \prime}}( \mathcal{A}^{T}_{Q,W,\bn-\bm}) \]  for all $\b0 < \bm < \bn$, taking cup product with $\mathrm{eu}^{T}_{\bm,\bn-\bm}(I)/\mathrm{eu}^{T}_{\bm,\bn-\bm}(\edge)$ gives that \[ \Delta_{\bm,\bn-\bm} (\gamma) \in P^{d}(\mathcal{A}^{T}_{Q,W} \tilde{\otimes} \mathcal{A}^{T}_{Q,W}).\] Thus, $\gamma$ is a primitive element in the associated graded algebra. By Proposition \ref{prop:primitivecalculation}, this means that $\gamma = \beta u^k$ for some element $\beta \in P^1(\mathcal{A}^T_{Q,W})$. However, since \eqref{eqn:propertyCoHA2} holds, we have $k \leq \frac{d-1}{2}$ using Proposition \ref{partialuder}. Thus, we conclude \eqref{eqn:propertyCoHA1}.

\end{proof}

\medskip 

\subsection{Uniqueness}

\medskip 

\begin{proof} \emph{of Theorem \ref{thm:main}:}
When $W=0$, by Example \ref{ex:any}, $P^{\bullet}$ and $F^{\bullet}$ define two filtrations on $\mathcal{A}^{T}_Q$. Notice that for any vertex $i$, we have
$$
P^{d}(\mathcal{A}^{T}_{Q,\delta_i}) = \mathrm{span}(1,z_{i1},\cdots, z_{i1}^{\lfloor \frac{d-1}{2} \rfloor}) = F^d(\mathcal{A}^{T}_{Q,\delta_i}).
$$
Given a color symmetric function $E(z_{i1},\dots,z_{in_i})_{i \in I}$, we have by Proposition \ref{prop: actionwhenWiszero} \[ \mathrm{act}^{*}(E(z_{i1},\dots,z_{in_i})_{i \in I}) =  E(y+z_{i1},\dots,y+z_{in_i})_{i \in I}  \] and by definition \[ \oplus^{*}_{\bm,\bn-\bm}( E(z_{i1},\dots,z_{in_i})_{i \in I}) = E(z_{i1},\dots,z_{im_i}|z_{i,m_i+1},\dots,z_{in_i})_{i \in I} .\] Thus, recursively in $\bn \in \nn$, properties \eqref{eqn:propertyCoHA2} and \eqref{eqn: propertyCoHA3} are the same as \eqref{eqn:recursion 2} and \eqref{eqn:recursion 3}. Since these properties uniquely determine the filtration, it follows that $P^d(\mathcal{A}^{T}_{Q}) = F^{d}(\mathcal{A}^{T}_Q)$ and we are done. 

\end{proof}

\medskip

Let $F^d(\mathcal{A}^{T,\psi}_Q)$ denote the filtration on $\mathcal{A}^{T,\psi}_Q$ induced by isomorphism $\mathcal{V} \simeq \mathcal{A}^T_{Q}$. Note that Theorem \ref{thm:main} combined with the PBW Theorem \ref{integrality theorem} give another proof of Proposition \ref{prop:sub lie algebra}. We also note the following interesting consequence.

\medskip 

\begin{corollary} \label{BPS Lie algebra}
    There is an isomorphism of Lie algebras \[F^1(\mathcal{A}^{T,\psi}_Q) \simeq \mathfrak{g}^T_Q\]
\end{corollary}

\medskip 

\subsection{The morphism to the shuffle algebra} \label{morphismshuffle}

Let $Q$ be any symmetric quiver and $W$ be any $T$ invariant potential. Assume that there is another weighting $\textbf{w}: E \rightarrow \mathbb{Z}_{\geq 0}$ for which $W$ is homogenous and of strictly positive weight. Following \cite[Section 4.4]{botta2023okounkovs}, we construct a morphism of algebras 
\begin{equation}
\label{eqn:psi}
\Psi \colon \mathcal{A}^{T}_{Q,W} \rightarrow \mathcal{A}^{T}_{Q}.
\end{equation}
Let $Z^{T}_{\bn}(Q) := \Tr(W)_{\bn}^{-1}(0)$ and let $i_{\bn,T}: Z^T_{\bn}(Q) \rightarrow \mathfrak{M}^{T}_{\bn}(Q) $ be the inclusion of the zero set of the function 
$\Tr(W)_{\bn} \colon \mathfrak{M}^{T}_{\bn}(Q) \rightarrow \BC.$ From the definition of the vanishing cycle functor, we have a natural morphism of constructible complexes
\begin{equation} \label{eqn:vanishingmorphism} s \colon \varphi_{\Tr(W)} \BQ^{\vir}_{\mathfrak{M}^{T}_{\bn}(Q)} \rightarrow (i_{\bn,T})_{*}(i_{\bn,T})^{*} \BQ^{\vir}_{\mathfrak{M}^{T}_{\bn}(Q)}\end{equation} Taking global sections induces a morphism
\[ \HH(s): \mathcal{A}^{T}_{Q,W,\bn} \rightarrow \HH(Z^{T}_{\bn}(Q), i_{\bn,T}^{*}\mathbb{Q}^{\vir}_{\mathfrak{M}^{T}_{\bn}(Q)}). \]
On other hand, we have the restriction morphism of constructible complexes given by pullback along $i_{\bn,T}$:
\begin{equation} 
r \colon \label{eqn:pullbackalongclosed}\mathbb{Q}^{\vir}_{\mathfrak{M}^{T}_{\bn}(Q)} \rightarrow (i_{\bn,T})_{*}(i_{\bn,T})^{*} \mathbb{Q}^{\vir}_{\mathfrak{M}^{T}_{\bn}(Q)}  
\end{equation}
Taking global section induces s morphism \[ \HH(r): \HH(\mathfrak{M}^{T}_{\bn}(Q),\mathbb{Q}^{\vir}_{\mathfrak{M}^{T}_{\bn}(Q)})  \rightarrow  \HH(Z^{T}_{\bn}(Q),i_{\bn,T}^{*} \mathbb{Q}^{\vir}_{\mathfrak{M}^{T}_{\bn}(Q)})\] which we claim is an isomorphism. This is because if $Q^{\prime} \subset Q$ is the full subquiver consisting of exactly those arrows for which $\mathbf{w}(a)=0$ then by consider the action of torus induced by weighting $\mathbf{w}$, we can show that $Z^{T}_{\bn}(Q)$ is $T$ invariantly homotopic to $\mathfrak{M}^T_{\bn}(Q^{\prime})$, which in turn is homotopic to $\mathfrak{M}^{T}_{\bn}(Q)$. Thus, the restriction morphism $r$ is an isomorphism. We conclude that $r^{-1}s$ induces a morphism 
$$
\Psi_{\bn} \colon \mathcal{A}^{T}_{Q,W,\bn} \rightarrow \mathcal{A}^{T}_{Q,\bn}
$$ 
for all $\bn \in \mathbb{N}^{I}$, and taking the direct sum over $\bn$ induces the morphism $\Psi$ of \eqref{eqn:psi}. It was proved in \cite[Proposition 4.4]{botta2023okounkovs} that $\Psi$ is an algebra homomorphism. Let us assume that $Q,W,T$ satisfy Assumption \ref{assumption1}. We now show that $\Psi$ respects the perverse filtration. 

\medskip

\begin{proposition}
\label{prop:perverse}
The morphism $\mathcal{A}_{Q,W}^T\xrightarrow{\Psi} \mathcal{A}^T_Q$ respects the perverse filtrations:
\begin{equation}
\label{eqn:perverse preserve}
\Psi \left(P^d(\mathcal{A}_{Q,W}^T) \right) \subseteq P^d(\mathcal{A}_{Q}^T)
\end{equation}
for all $d \geq 1$.

\end{proposition}

\medskip

\begin{proof}
   It is proved in \cite[Proposition 17.4]{J} that the following diagram commutes \[\begin{tikzcd} \label{diag1} \tag{C1}
	{\mathcal{A}^{T}_{Q,W,\bn}} & {\mathcal{A}^{T}_{Q,\bn}} \\
	{\mathrm{H}( \BCu) \otimes \mathcal{A}^T_{Q,W,\bn} } & {\mathrm{H}( \BCu) \otimes \mathcal{A}^{T}_{Q,\bn}}
	\arrow["\Psi", from=1-1, to=1-2]
	\arrow["{\mathrm{act}^{*}}"', from=1-1, to=2-1]
	\arrow["{\mathrm{act}^*}", from=1-2, to=2-2]
	\arrow["{\mathrm{id} \otimes \Psi_{\bn}}"', from=2-1, to=2-2]
\end{tikzcd}\]
We claim that the morphism $\Psi$ also respects the direct sum pullback, i.e for any $\b0 < \bm < \bn$ the following diagram commutes
\[\begin{tikzcd} \label{diag2} \tag{C2}
		{{\mathcal{A}^{T}_{Q,W,\bn}}} & {\mathcal{A}^{T}_{Q,\bn}} \\
	{\mathcal{A}^T_{Q,W,\bm}  \otimes \mathcal{A}^T_{Q,W,\bn-\bm} } & {\mathcal{A}^T_{Q,\bm}  \otimes \mathcal{A}^{T}_{Q,\bn-\bm}}
	\arrow["{\Psi_{\bn}}", from=1-1, to=1-2]
	\arrow["{{\oplus^{*}_{\bm,\bn-\bm}}}"', from=1-1, to=2-1]
	\arrow["{{\oplus^{*}_{\bm,\bn-\bm}}}", from=1-2, to=2-2]
	\arrow["{{\Psi_{\bm} \otimes \Psi_{\bn-\bm}}}", from=2-1, to=2-2]
\end{tikzcd}\]
Note that the following diagram commutes
\[\begin{tikzcd}
	{Z^{T}_{\bm}(Q) \times_{\BT}  Z^{T}_{\bn-\bm}(Q)} & {Z^T_{\bn}(Q)} \\
	{\mathfrak{M}^{T}_{\bm}(Q) \times_{\BT}  \mathfrak{M}^{T}_{\bn-\bm}(Q)} & {\mathfrak{M}^{T}_{\bn}(Q)}
	\arrow["{{\oplus }}", from=1-1, to=1-2]
	\arrow["{{i_{\bm,T} \times_{\BT} i_{\bn-\bm,T}}}"', from=1-1, to=2-1]
	\arrow["{i_{\bn,T}}", from=1-2, to=2-2]
	\arrow["{\oplus }", from=2-1, to=2-2]
\end{tikzcd}\]
Ignoring cohomological shifts, the naturality of morphism $s$ and Thom-Sebastiani implies that the the following diagram commutes
\[\begin{tikzcd}
	{\varphi_{\Tr(W)} \BQ^{\vir}_{\mathfrak{M}^{T}_{\bn}(Q)}} & {\varphi_{\Tr(W)} \BQ^{\vir}_{\mathfrak{M}^{T}_{\bm}(Q)}\boxtimes_{\oplus} \varphi_{\Tr(W)} \BQ^{\vir}_{\mathfrak{M}^{T}_{\bn-\bm}(Q)}} \\
	{(i_{\bn,T})_{*}(i_{\bn,T})^{*} \mathbb{Q}^{\vir}_{\mathfrak{M}^{T}_{\bn}(Q)}} & {(i_{\bn,T})_{*}(i_{\bn,T})^{*} \mathbb{Q}^{\vir}_{\mathfrak{M}^{T}_{\bm}(Q)} \boxtimes_{\oplus} (i_{\bn-\bm,T})_{*}(i_{\bn-\bm,T})^{*} \mathbb{Q}^{\vir}_{\mathfrak{M}^{T}_{\bn}(Q)}} \\
	{\mathbb{Q}^{\vir}_{\mathfrak{M}^{T}_{\bn}(Q)} } & {\mathbb{Q}^{\vir}_{\mathfrak{M}^{T}_{\bm}(Q)} \boxtimes \mathbb{Q}^{\vir}_{\mathfrak{M}^{T}_{\bm}(Q)} }
	\arrow["{\oplus^*}", from=1-1, to=1-2]
	\arrow["r", from=1-1, to=2-1]
	\arrow["{r \times r}", from=1-2, to=2-2]
	\arrow["{\oplus^{*}}", from=2-1, to=2-2]
	\arrow["s", from=3-1, to=2-1]
	\arrow["{\oplus^{*}}", from=3-1, to=3-2]
	\arrow["{s \times s}", from=3-2, to=2-2]
\end{tikzcd}\]
which by taking global section implies the commutativity of \eqref{diag2}.

\medskip 

We will now prove \eqref{eqn:perverse preserve} by induction on $\bn \in \nn$. If $\alpha \in P^d(\mathcal{A}^{T}_{Q,W,\delta_i})$, then Proposition \ref{prop:key2} and the commutativity of the diagram \ref{diag1} imply that $\Psi(\alpha) \in P^d(\mathcal{A}^T_{Q,\delta_i})$. Now consider any $\alpha \in P^{d}(\mathcal{A}^{T}_{Q,W,\bn})$ for $\bn$ larger than $\{\bs^i\}_{i \in I}$. The commutativity of diagram \eqref{diag1} implies that 
$$
\text{deg}_y \left( \mathrm{act}^*(\Psi(\alpha)) \right) \leq \frac {d-1}2
$$
which is precisely the statement that $\Psi(\alpha)$ satisfies property \eqref{eqn:propertyCoHA2}. Similarly, by the induction hypothesis on $\bn$, the commutativity of diagram \eqref{diag2} implies that $\Psi(\alpha)$ satisfies property \ref{eqn: propertyCoHA3}. Thus by Proposition \ref{prop:key2}, $\Psi(\alpha) \in P^d(\mathcal{A}_Q)$ and we are done. 

\end{proof}

Although the morphism $\Psi$ is generally well behaved, it fails to be injective when the monodromy on vanishing cycles cohomology is non-trivial. In those cases when it is injective, we now show that then in fact $\Psi$ is strict, in the sense of \eqref{eqn:strict}. This will allow us to compare $\fg^T_{Q,W}$ with $\fg^T_Q$. 

\medskip 

\begin{proposition}

If $\Psi$ is injective, then
    \begin{equation}
    \label{eqn:strict}
    \Psi(P^d(\mathcal{A}^T_{Q,W})) = \mathrm{Im}(\Psi) \cap P^d(\mathcal{A}^T_{Q}) 
    \end{equation}
    for all $d \geq 1$.
\end{proposition}

\begin{proof}
Since the morphism $\Psi$ respects the perverse filtration, it induces a morphism of algebras $\mathrm{Gr}(\Psi) \colon \mathrm{Gr}(\mathcal{A}_{Q,W}^T) \rightarrow \mathrm{Gr}(\mathcal{A}_Q^T)$. The morphism $\Psi$ also respects the action of the $u$ operator by \cite[Proposition 14.1]{J}. Thus, $\Psi$ restricts to an injection of vector spaces $\Psi:\fg^T_{Q,W}[u] \rightarrow \fg^T_{Q}[u]$. This induces an injection of $\HHT$ modules $\Sym(\fg^T_{Q,W}[u]) \rightarrow \Sym(\fg^T_{Q}[u])$ and thus $\Gr(\Psi)$ is an injection by the cohomological integrality theorem (Proposition \ref{integrality theorem}). 

\end{proof}

\medskip

\section{Tripled quivers and the cubic potential}
\label{sec:triple}

We will now consider the particular case of a tripled quiver $\tQ$ with the canonical cubic potential, as in Examples \ref{ex:triple} and \ref{ex:cubic}, and derive some consequences of our results in this case. We set $R = \HHT$ endowed with particular elements $\{\hbar,u_{\alpha}\}_{\alpha \in \edge}$.

\medskip

\subsection{Wheel conditions}
\label{sub:wheel}

We henceforth only consider the shuffle algebra of Example \ref{ex:triple shuffle}, with $\CV$ referring exclusively to the shuffle algebra defined with respect to the zeta function \eqref{eqn:zeta triple}. We recall that the latter is given by
$$
\zeta_{ij}(x) = \left( \frac {x-\hbar}x \right)^{\delta_{ij}} \prod_{\alpha: \oij} (-x-u_\alpha) \prod_{\alpha:\oji}(-x-\hbar+u_\alpha)
$$
In the present paper, two assumptions will arise pertaining to the parameters $\{\hbar,u_{\alpha}\}_{\alpha \in \edge}$. The weak one will be in force throughout the present Section, and the strong one will be in force only when explicitly invoked.

\medskip 

\noindent \textbf{Weak:} There exists a character $\chi : M \rightarrow \BZ$ of the finitely generated abelian group $M = \text{span}_{\BZ}(\hbar,u_\alpha)_{\alpha \in \edge} \subset R$ such that
$$
\chi(\hbar) > \chi(u_\alpha) > 0
$$
for all $\alpha \in \edge$. Geometrically, this corresponds to the CoHA of Example \ref{ex:triple} for a torus $T$ that contains a one-parameter subgroup which acts on all arrows of the doubled quiver with strictly positive weight. 

\medskip 

\noindent \textbf{Strong:} There exist characters $\chi_1,\chi_2$ of the finitely generated abelian group $M = \text{span}_{\BZ}(\hbar,u_\alpha)_{\alpha \in \edge} \subset R$ such that
$$
\chi_1(u_\alpha) = 1 , \quad \chi_1(u_{\alpha^*}) = 0
$$
$$
\chi_2(u_\alpha) = 0 , \quad \chi_2(u_{\alpha^*}) = 1
$$
for all $\alpha \in \edge$. Geometrically, this corresponds to the CoHA of Example \ref{ex:triple} for a torus $T$ that contains a two-parameter subgroup which acts on all arrows of the original quiver with one weight, and on all the reverse arrows with another weight.

\medskip

\begin{definition}
\label{def:wheel}

As in \cite[Definition 5.2]{NWheel}, consider the subset
$$
\CS \subset \CV 
$$
consisting of polynomials $E(z_{ia})_{i \in I, a \geq 1}$ satisfying the so-called wheel conditions:
\begin{equation}
\label{eqn:wheel general}
\prod_{\alpha : \oij} (z_{jb} - z_{ic} - u_\alpha) \prod_{\alpha : \oji} (z_{jb} - z_{ic} - \hbar + u_\alpha)  \quad \text{divides} \quad E \Big|_{z_{ia} = z_{ic} + \hbar} 
\end{equation}
for all $i,j \in I$ and $a \neq c$ (and further $a \neq b \neq c$ if $i = j$).

\end{definition}

\medskip 

\begin{remark} 

If all the elements $\{u_\alpha, \hbar-u_{\alpha'}\}_{\alpha : i \rightarrow j, \alpha' : j \rightarrow i}$ are distinct, then conditions \eqref{eqn:wheel general} boil down to
\begin{equation}
\label{eqn:wheel}
E \Big|_{z_{ia} = z_{jb} + \hbar - u_\alpha = z_{ic} + \hbar} = E \Big|_{z_{ja} = z_{ib} + u_\alpha = z_{jc} + \hbar} = 0
\end{equation}
for any arrow $\alpha : \oij$ and all $a,b,c$ such that $a \neq c$ (and further $a \neq b \neq c$ if $i = j$). Relations \eqref{eqn:wheel} are the celebrated Feigin-Odesskii wheel conditions. 

\end{remark}

\medskip 

\noindent It is straightforward to see that $\CS$ is a subalgebra of $\CV$ with respect to the shuffle product, and that it contains the \emph{spherical} subalgebra
$$
\oCS \subseteq \CV
$$
defined as the $\BK$-span of
$$
\Big\{z_{i_11}^{d_1} * \dots * z_{i_n1}^{d_n} \Big\}_{i_1,\dots,i_n \in I, d_1,\dots,d_n \geq 0}
$$
We will write $\CV_{\text{loc}}, \CS_{\text{loc}}, \oCS_{\text{loc}}$ for $\CV,\CS,\oCS$ tensored with $\text{Frac}(R)$. The following was implicitly conjectured in \cite{NGen}, see \cite{NWheel} in the $K$-theoretic case.

\medskip 

\begin{conjecture} 
\label{conj:wheel}

We have $\CS_{\emph{loc}} = \oCS_{\emph{loc}}$. 

\end{conjecture} 

\medskip

\subsection{The CoHA of the tripled quiver} 

Consider the CoHA from Example \ref{ex:cubic} and the map to the shuffle algebra
\begin{equation}
\label{eqn:coha 1}
\Psi : \CA^T_{\tQ,\tW} \longrightarrow \CA^T_{\tQ} = \CV
\end{equation}
Under the strong assumption in Subsection \ref{sub:wheel}, the map $\Psi$ was shown to be injective in \cite[Theorem 10.2]{davison2022integrality} and \cite[Proposition 4.6]{schiffmanncohagenerators}. We will now show that the image of the map above lies in the subalgebra $\CS$ of Definition \ref{def:wheel}, by adapting to cohomology the arguments of \cite[Theorem 2.9, Corollary 2.10]{Zhao} and \cite[Proposition 2.11]{NWheel}.

\medskip 

\begin{proposition}
\label{prop:yu}

We have 
\begin{equation}
\label{eqn:coha 2}
\Psi\left(\CA^T_{\tQ,\tW}\right) \subseteq \CS
\end{equation}

\end{proposition}

\medskip 

\begin{proof} Fix $\bn \in \nn$ and consider the affine space
\begin{equation}
\label{eqn:x}
X_{\bn} = \text{Rep}_{\bn}(\bar{Q}) = T^*\text{Rep}_{\bn}(Q) 
\end{equation}
whose points parameterize collections of linear maps
\begin{equation}
\label{eqn:point}
\left( \BC^{n_i} \xleftrightharpoons[\phi_\alpha]{\phi_\alpha^*} \BC^{n_j} \right)_{\alpha : \oij} 
\end{equation}
The algebraic group $\GL_{\bn} = \prod_{i \in I} \GL_{n_i}$ acts on $X_{\bn}$ via conjugation, and preserves the closed substack  
\begin{equation}
\label{eqn:z}
\Pi_{\bn} = \left\{ \sum_{\alpha \in \edge} \phi_\alpha \phi_\alpha^* - \phi_\alpha^*\phi_\alpha = 0 \right\} \stackrel{\iota}\hookrightarrow X_{\bn}
\end{equation}
We will consider a maximal torus $A \subset \GL_{\bn}$, which determines a usual basis of the vector spaces $\BC^{n_i}$, indexed by $a \in \{1,\dots,n_i\}$. We will write $M_{ba}$ for the linear map $\BC^{n_i}\rightarrow \BC^{n_j}$ which sends the $a$-th basis vector in the domain to the $b$-th basis vector in the codomain. Moreover, we will write
$$
\{z_{i1},\dots,z_{in_i}\}_{i \in I}
$$
for the elementary characters of the maximal torus $A$, so that $\HH^{\GL_{\bn} \times A}$ consists of color-symmetric polynomials in the $z_{ia}$'s. It is well known (see for instance \cite[Proposition 4.5]{botta2023okounkovs}) that $\text{Im }\Psi$ lies inside the image of
$$
\iota_* : \HH^{T \times \GL_{\bn}}_{\text{BM}}(\Pi_{\bn}) \rightarrow \HH^{T \times \GL_{\bn}}(X_{\bn}) = \CV_{\bn}
$$
Therefore, it suffices to show that for any $\beta \in \HH^{T \times \GL_{\bn}}_{\text{BM}}(\Pi_{\bn})$, the polynomial $\iota_*(\beta)$ satisfies the wheel conditions \eqref{eqn:wheel general}. To this end, for any $i,j \in I$ and any $\gamma \in \text{span}_{\BZ}(\hbar, u_\alpha)_{\alpha \in \edge}$, let us consider all arrows: 
$$
\alpha_1,\dots,\alpha_d : \oij \quad \text{and} \quad \alpha_1',\dots,\alpha'_{d'} : \oji
$$
such that:
$$
u_{\alpha_1} = \dots = u_{\alpha_d} = \hbar - u_{\alpha'_1} = \dots = \hbar - u_{\alpha'_{d'}} = \gamma
$$
We need to show that for any $\beta \in \HH^{T \times \GL_{\bn}}_{\text{BM}}(\Pi_{\bn})$, the shuffle element $E = \iota_*(\beta)$ has the property that
\begin{equation}
\label{eqn:division property}
(z_{jb} - z_{ic} - \gamma)^{d+d'} \quad \text{divides} \quad E \Big|_{z_{ia} = z_{ic} + \hbar}
\end{equation}
(for any indices $a\neq c$, such that moreover $a\neq b \neq c$ if $i=j$). Consider the locally closed subset $\rho : V_\gamma \hookrightarrow X_{\bn}$ of points \eqref{eqn:point} which take the following form:
$$
(\phi_\alpha,\phi_\alpha^*) = \begin{cases} (x_r M_{bc}, y_r M_{ab}) &\text{if }\alpha = \alpha_r \text{ for some }r \in \{1,\dots,d\} \\ (y_{r'}' M_{ab}, x'_{r'} M_{bc}) &\text{if }\alpha = \alpha'_{r'} \text{ for some }r' \in \{1,\dots,d'\} \\ (0,0) &\text{otherwise} \end{cases}
$$
where in the formula above the $x$'s and $y$'s are complex numbers which satisfy
$$
x_1y_1+\dots+x_dy_d - x_1'y_1' - \dots - x_{d'}'y_{d'}' \neq 0
$$
Thus, the set $V_\gamma$ is the complement of a hypersurface in the affine space $X_{\bn}$ with coordinates $x_r,y_r, x'_{r'}, y_{r'}'$, and the action $T \times A \curvearrowright V_\gamma$ is such that the $x$'s and $y$'s are rescaled by the characters:
$$
z_{jb}- z_{ic}- \gamma \qquad \text{and} \qquad z_{ia}-z_{jb}+\gamma- \hbar
$$
respectively. Note that $V_{\gamma} \cap \text{Im }\iota_* = \varnothing$, so $E = \iota_*(\beta)$ implies $\rho^*(E) = 0$. We will use the latter equation to obtain \eqref{eqn:division property}. Replacing $V_{\gamma}$ by the closed subset
$$
\bar{V}_\gamma = \Big\{x_1y_1+\dots+x_dy_d - x_1'y_1' - \dots - x_{d'}'y_{d'}' = 1 \Big\}
$$
has the effect of replacing $E$ by $\bar{E} = E|_{z_{ia} = z_{ic} + 
\hbar}$. However, $\bar{V}_\gamma$ is an affine bundle over projective space with coordinates $x_1,\dots,x_d,x_1',\dots,x'_{d'}$, so its equivariant cohomology is isomorphic to that of projective space, namely:
$$
\HH^T[\dots,z_{ia},\dots,z_{jb},\dots,z_{ic}, \dots]\Big|_{z_{ia} = z_{ic} + \hbar} \Big/ (z_{jb} - z_{ic} - \gamma)^{d+d'}
$$
(the particular quotient is due to the fact that the coordinates $x_1,\dots,x_d,x_1',\dots,x'_{d'}$ are all rescaled by the equivariant parameter $z_{jb} - z_{ic} - \gamma$). Then the fact $\rho^*(E) = 0$ precisely implies \eqref{eqn:division property}.

\end{proof}

\medskip 

\subsection{Spherical generation}

We will now consider the localized CoHA, i.e.
$$
\CA^T_{\tQ,\tW,\mathrm{loc}} = \CA^T_{\tQ,\tW} \bigotimes_{R} \text{Frac}(R)
$$
It was shown in \cite{NGen} that the localized CoHA is spherically generated, hence
\begin{equation}
\label{eqn:coha 3}
\Psi : \CA^T_{\tQ,\tW,\mathrm{loc}} \xrightarrow{\sim} \oCS_{\text{loc}}
\end{equation}
under the strong assumption of Subsection \ref{sub:wheel}. Therefore, Conjecture \ref{conj:wheel} is equivalent to the following opposite inclusion to \eqref{eqn:coha 2}.

\medskip 

\begin{conjecture} 
\label{conj:coha}

The localization of the map \eqref{eqn:coha 2} yields an isomorphism
$$
\Psi : \CA^T_{\tQ,\tW,\emph{loc}} \xrightarrow{\sim} \CS_{\emph{loc}}
$$
under the strong assumption of Subsection \ref{sub:wheel}. 

\end{conjecture} 

\medskip 

Theorem \ref{thm:main} and \ref{hommorphismcompatibility} have the following consequence.

\medskip 

\begin{corollary} 
\label{cor:tripled}

Assuming Conjecture \ref{conj:coha}, the perverse filtration on the localized CoHA of the tripled quiver with canonical cubic potential is explicitly described as
$$
P^d  \CA^T_{\tQ,\tW,\emph{loc}} \cong \Big\{\text{polynomials satisfying \eqref{eqn:degree bound} and \eqref{eqn:wheel general}}\Big\}
$$
under the strong assumption of Subsection \ref{sub:wheel}. 

\end{corollary}

\medskip 

\noindent In particular, setting $d = 1$ would imply that $\mathfrak{g}^{T}_{\tilde{Q},\tilde{W},\mathrm{loc}}$ can be explicitly realized as the (sub super Lie algebra) of $\CV$ consisting of $E(z_{i1},\dots,z_{in_i})_{i \in I}$ for which

\medskip 

\begin{itemize}

\item the wheel conditions \eqref{eqn:wheel general} hold

\medskip

\item for any partition of the degree vector $\bn = (n_i)_{i \in I}$ 
$$
\bn = \bn^1 + \dots + \bn^k
$$
(for arbitrary $k \geq 1$ and $\bn^1,\dots,\bn^k \in \nn \backslash \b0$), the polynomial
$$
E(y_1+z_{i,1},\dots,y_1+z_{i,n_i^1}, \dots,y_k+z_{i,n_i-n_i^k+1}, \dots, y_k+z_{i,n_i})_{i \in I}
$$
has total degree in $y_1,\dots,y_k$ at most $\frac {1-k}2 - \sum_{1\leq a < b \leq k} (\bn^a,\bn^b)$.

\end{itemize}

\medskip 

\subsection{The preprojective CoHA}

The preprojective CoHA refers to
$$
\CA_{\Pi_Q} = \bigoplus_{\bn \in \nn} \HH^T(\Pi_{\bn})
$$
where $\Pi_{\bn} \subset X_{\bn}$ are defined in \eqref{eqn:x} and \eqref{eqn:z}, endowed with the Schiffmann-Vasserot cohomological Hall algebra structure (see \cite{Yang_2018} for details). It was shown in  \cite[Appendix]{Ren:2015zua} and \cite{Yang_2019} that we have isomorphisms of algebras 
$$
\CA^T_{\tQ,\tW} \cong \CA_{\Pi_Q} 
$$ 
and in \cite[Theorem 4.1]{davison2022bps} that there is isomorphism of Lie algebras 
$$
\mathfrak{g}^{T}_{\tilde{Q},\tilde{W}} \simeq \mathfrak{g}^{T}_{\Pi_Q}
$$ 
where $\CA_{\Pi_Q}$ denotes the cohomological Hall algebra of Schiffmann-Vasserot and Yang-Zhao, and $\mathfrak{g}^{T}_{\Pi_Q}$ denotes its BPS Lie algebra (defined as the cohomology of the BPS sheaf $\mathcal{BPS}_{\Pi_Q}$ for the preprojective algebra, which is in turn defined by considering the support of BPS sheaf $\mathcal{BPS}_{\tilde{Q},\tilde{W}}$). Then Corollary \ref{cor:tripled} (which is contingent on Conjecture \ref{conj:coha}) would also provide an explicit description of localized preprojective BPS Lie algebra $\mathfrak{g}^T_{\Pi_Q,\mathrm{loc}}$. Recall that the following formula for the graded dimension of the localized preprojective BPS Lie algebra was proved in \cite{davisonkac} and \cite{schiffmanncohagenerators}
$$
\text{grdim}_{\BK} \mathfrak{g}^{T}_{\Pi_Q,\mathrm{loc}} = \sum_{\bn \in \nn} q^{\bn} A_{Q,\bn}(1)
$$
where $A_{Q,\bn}$ denotes the Kac polynomial of the quiver $Q$ in dimension $\bn$. Therefore, Corollary \ref{cor:tripled} would imply that the vector space of polynomials satisfying the conditions in the two bullets in the previous Subsection has dimension equal to $A_{Q,\bn}(1)$. This statement would be the cohomological version of \cite[Conjecture 4.5]{NR}, since the algebra $\CB_0$ of \loccit is expected to $q$-deform $\mathfrak{g}^{T}_{\Pi_Q,\mathrm{loc}}$.

\medskip 

\printbibliography

\end{document}